\documentclass[10pt,reqno,a4paper]{amsart}
\usepackage{amsfonts,amsmath,amssymb,times}
\usepackage{a4wide}
\usepackage{color}

\ifx\pdfoutput\undefined
\usepackage[dvips]{graphicx}
\else
\usepackage[pdftex]{graphicx}
\definecolor{blau}{rgb}{0,0,0.6} 
\usepackage{hyperref}
\hypersetup{colorlinks,linkcolor=blau,citecolor=blue, urlcolor=blue}

\allowdisplaybreaks

\newtheorem{theorem}{Theorem}
\newtheorem{lemma}{Lemma}

\newtheorem{coroll}{Corollary}
\theoremstyle{definition}
\newtheorem{remark}{Remark}

\def\P{{\mathbb {P}}}
\def\E{{\mathbb {E}}}

\DeclareMathOperator{\Exp}{Exp}

\newcommand{\fallfak}[2]{\ensuremath{#1^{\underline{#2}}}}

\newcommand{\claw}{\ensuremath{\xrightarrow{(d)}}}
\newcommand{\law}{\ensuremath{\stackrel{(d)}{=}}}
\newcommand{\Z}{\ensuremath{\mathbb{Z}}}
\newcommand{\N}{\ensuremath{\mathbb{N}}}
\newcommand{\R}{\ensuremath{\mathbb{R}}}
\newcommand{\C}{\ensuremath{\mathbb{C}}}
\newcommand{\Stir}[2]{\genfrac{ \{ }{ \} }{0pt}{}{#1}{#2}}
\newcommand{\stir}[2]{\genfrac{ [ }{ ] }{0pt}{}{#1}{#2}}

\title{On Sampling without replacement and OK-Corral urn models}

\author[M.~Kuba]{Markus Kuba}
\address{Markus Kuba\\
Institut f{\"u}r Diskrete Mathematik und Geometrie\\
Technische Universit\"at Wien\\
Wiedner Hauptstr. 8-10/104\\
1040 Wien -- HTL Wien 5 Spengergasse, Spengergasse 20, 1050 Wien, Austria} %
\email{kuba@dmg.tuwien.ac.at}

\thanks{This work was supported by the Austrian Science Foundation FWF, S9608-N13.}

\date{\today}


\begin{document}

\begin{abstract}
In this work we discuss two urn models with general weight sequences $(A,B)$ associated to them, $A=(\alpha_n)_{n\in\N}$ and $B=(\beta_m)_{m\in\N}$, generalizing two well known 
P\'olya-Eggenberger urn models, namely the so-called sampling without replacement urn model and the OK Corral urn model.
We derive simple explicit expressions for the distribution of the number of white balls, when all black have been drawn, and obtain as a byproduct the corresponding results
for the P\'olya-Eggenberger urn models. Moreover, we show that the sampling without replacement urn models and the OK Corral urn models with general weights
are dual to each other in a certain sense. We discuss extensions to higher dimensional sampling without replacement and OK Corral urn models, respectively, where we also obtain explicit results for the probability mass functions, and also an analog of the dualitiy relation. Finally, we derive limit laws for a special choice of the weight sequences.

\end{abstract}
\maketitle

\section{Introduction}
In this work we analyse two sequential urn models with two types of balls. At the beginning, the urn
contains $n$ white and $m$ black balls, with $n,m\ge 1$. Associated to the urn model 
are two sequences of positive real numbers $A=(\alpha_n)_{n\in\N}$ and $B=(\beta_m)_{m\in\N}$. At every step, we draw a ball from the urn according to the number of white and black balls present in the urn with respect to the sequences $(A,B)$, subject to the two models defined below. The choosen ball is discarded and the sampling procedure continues until one type of balls is completely drawn.\\[0.2cm] 
\emph{Urn model I (Sampling with replacement with general weights)}. Assume that $n$ white and $m$ black balls are contained in the urn, with arbitrary $n,m\in \N$. A white ball is drawn with probability $\alpha_n/(\alpha_n+\beta_m)$, and a black ball is drawn with probability $\beta_m/(\alpha_n+\beta_m)$. Additionally, we assume for urn model I that $\alpha_0=\beta_0=0$. \\[0.2cm]
\emph{Urn model II (OK Corral urn model with general weights)}. For arbitrary $n,m\in \N$ assume that $n$ white and $m$ black balls are contained in the urn. A white ball is drawn with probability $\beta_m/(\alpha_n+\beta_m)$, and a black ball is drawn with probability $\alpha_n/(\alpha_n+\beta_m)$. \\[0.2cm]
The absorbing states, i.e.~the points where the evolution of the urn models stop, are given for both urn models by the positive lattice points on the the coordinate axes $\{(0,n)\mid n\ge 1\} \cup \{(m,0)\mid n\ge 1\}$. These two urn models generalize two famous P\'olya-Eggenberger urn models with two types of balls, namely the classical sampling without replacement (I), and the so-called OK-Corral urn model (II), described in detail below. We are interested in the distribution of the random variable $X_{m,n}$ counting the number of white balls, when all black balls have been drawn. 
For both urn models we will derive explicit expression for the probability mass function $\P\{X_{n,m}=k\}$ of the random variable $X_{n,m}$. Moreover, we will describe a certain duality between 
the sampling without replacement urn model and the OK Corral urn model using an appropriate transformation of weight sequences. 

\subsection{Weighted lattice paths and urn models}
Combinatorially, we can describe the evolution of the urn models I and II by weighted
lattice paths. If the urn contains $n$ white balls and $m$ black balls and we select a white ball, then
this corresponds to a step from $(m,n)$ to $(m,n-1)$, to which the
weight $\alpha_n/(\alpha_n+\beta_m)$ is associated for urn model I, and $\beta_m/(\alpha_n+\beta_m)$ for urn model II. 
Converserly, if we select a black ball, this corresponds to a step from
$(m,n)$ to $(m-1,n)$, to which the
weight $\beta_m/(\alpha_n+\beta_m)$ is associated for urn model I, and $\alpha_n/(\alpha_n+\beta_m)$ for urn model II. 
The weight of a path after $t$ successive draws consists of the product of the
weight of every step. By this correspondence, the probability of
starting at $(m,n)$ and ending at $(j,k)$ is equal to the sum of the
weights of all possible paths starting at state $(m,n)$ and ending
at state $(j,k)$, for $j,k\in\N_0$. The expressions for the probability, obtained by this combinatorial model are, although exact, very involved
and not useful for i.e.~the calculation of the moments of $X_{n,m}$, or for deriving asymptotic expansions for large $n$ and $m$.

\subsection{Connection to ordinary P\'olya-Eggenberger urn models}
The dynamics of ordinary P\'olya-Eggen\-berger urn models are very similar to the dynamics of the two previously descirbed urn models.
At the beginning, the urn contains $n$ white and $m$ black balls. At every step, we choose a
ball at random from the urn - a white ball is drawn with probability $n/(n+m)$ and a black ball is drawn with probability $m/(n+m)$ - and examine its color and put it back into the urn and then add/remove balls according to its color by the
following rules. If the ball is white, then we put $a$ white and $b$
black balls into the urn, while if the ball is black, then $c$ white
balls and $d$ black balls are put into the urn. The values $a, b, c,
d \in \mathbb{Z}$ are fixed integer values and the urn model is
specified by the transition matrix $M = \bigl(\begin{smallmatrix} a & b \\
c & d\end{smallmatrix}\bigr)$. We can specialise the sequences $A=(\alpha_n)_{n\in\N}$ and $B=(\beta_m)_{m\in\N}$ in order to obtain 
two families of P\'olya-Eggenberger urn models as special cases.

\vspace{-0.15cm}

\subsection{Sampling without replacement and generalizations.} This classical example corresponds to the urn with transition matrix $M =
\bigl(\begin{smallmatrix} -1 & 0 \\ 0 & -1\end{smallmatrix}\bigr)$. In this model, balls are drawn one
after another from an urn containing balls of two different colors
and not replaced. What is the probability that $k$ balls of one
color remain when balls of the other color are all removed? 
The generalized sampling without replacement urn is specified by ball transition matrix given by $M = \bigl(\begin{smallmatrix} -a & 0 \\ 0 & -d\end{smallmatrix}\bigr)$, with $a,d\in\N$. In the general model one asks for the number of white balls when all black balls have been removed,
starting with $a\cdot n$ white balls and $d\cdot m$ black balls. Let $Y_{an,dm}$ denote the random variable counting the number of white balls when all black balls have been removed. 
The random variable $Y_{an,dm}$ is related to $X_{n,m}$ of urn model I in the following way.
\begin{equation*}
\frac{Y_{an,dm}}{a} =X_{n,m},\quad \text{for the special choices} \quad A=(\alpha_n)_{n\in\N}=(a\cdot n)_{n\in\N},\quad B=(\beta_m)_{m\in\N}=(d\cdot m)_{m\in\N}.
\end{equation*}
The most basic case $a=d=1$ can be treated by several different methods such as generating functions, lattice path counting arguments, etc. The distribution of $Y_{n,m}$ is a folklore result and given as follows.
\begin{equation*}
\P\{Y_{n,m}=k\}=\frac{\binom{n+m-1-k}{m-1}}{\binom{n+m}{n}}, \quad 0\le k\le n.
\end{equation*}

\vspace{-0.15cm}

\subsection{The OK Corral problem and generalizations}
The so-called OK Corral urn serves as a mathematical model of the historical gun fight at the OK Corral.
This problem was introduced by Williams and McIlroy in
\cite{WilMcI1998} and studied recently by several authors using
different approaches, leading to very deep and interesting results; see
\cite{Stad1998,Kin1999,Kin2002,KinVol2003,PuyPhD,FlaDumPuy2006}. Also the urn corresponding
to the OK corral problem can be viewed as a basic model in the mathematical theory of warfare and conflicts; see \cite{Kin2002,KinVol2003}.
An interpretation is given as follows. Two groups of gunmen, group A and group B (with $n$ and $m$
gunmen, respectively), face each other. At every discrete time step,
one gunman is chosen uniformly at random who then shoots and kills
exactly one gunman of the other group. The gunfight ends when
one group gets completely ``eliminated''. Several questions are of
interest: what is the probability that group A (group B)
survives, and what is the probability that the gunfight ends
with $k$ survivors of group A (group B)? 
\smallskip
Apparently, the described process corresponds to a P\'olya-Eggen\-berger urn model with transition matrix 
$M =\bigl(\begin{smallmatrix} 0 & -1 \\ -1 & 0\end{smallmatrix}\bigr)$. This model was analyzed by Williams and McIlroy~\cite{WilMcI1998}, who obtained an interesting result for the expected value of the number of survivors. Using martingale arguments and the method of moments Kingman~\cite{Kin1999} gave limiting distribution results for the OK Corral urn model
for the total number of survivors. Moreover, Kingman~\cite{Kin2002} obtained further results in a very general setting of Lanchester's theory of warfare. Kingman and Volkov~\cite{KinVol2003} gave a more detailed analysis of the balanced OK Corral urn model using an ingenious connection 
to the famous Friedman urn model; amongst others, they derived an explicit result for the number of survivors. 
In his Ph.~D.~thesis~\cite{PuyPhD} Puyhaubert, guided by Flajolet, extended the results of \cite{Kin1999,KinVol2003} on the balanced OK Corral urn model using analytic combinatoric methods concerning the number of survivors of a certain group. His study is based on the connection to the Friedman urn showed in~\cite{KinVol2003}. He obtained explicit expression for the probability distribution, the moments, and also reobtained (and refined) most of the limiting distribution results reported earlier. 
Some results of~\cite{PuyPhD} where reported in the work of Flajolet et al.~\cite{FlaDumPuy2006}, 
see also Flajolet~\cite{FlaTalk}. Stadje~\cite{Stad1998} obtained several limiting distribution results for the generalized OK Corral urn, as introduced below, and also for related urn models with more general transition probabilities. 

\smallskip

A natural generalization of the OK Corral urn model is to consider an urn model with transition matrix given by $M =\bigl(\begin{smallmatrix} 0 & -b \\ -c & 0\end{smallmatrix}\bigr)$, $b,c\in\N$, which may serve as a simple mathematical model of warfare with unequal fire power. Let $Y_{cn,bm}$ denote the random variable counting the number of white balls when all black balls have been removed, 
starting with $c\cdot n$ white balls and $b\cdot m$ black balls. The random variable $Y_{cn,bm}$ is related to the random variable $X_{n,m}$ of urn model II in the following way.
\begin{equation*}
\frac{Y_{cn,bm}}{c} =X_{n,m},\quad \text{for the special choices} \quad A=(\alpha_n)_{n\in\N}=(c\cdot n)_{n\in\N},\quad B=(\beta_m)_{m\in\N}=(b\cdot m)_{m\in\N}.
\end{equation*}

Concerning the special case $c=b=1$ Puyhaubert and Flajolet obtained the following result for the distribution of $Y_{n,m}$.
\begin{equation}
\label{eqnb14}
    \P\{Y_{n,m} = k\} = \frac{k!}{(n+m)!}\sum_{r=1}^{n} (-1)^{n-r} \binom{n+m}{n-r}\binom{r-1}{k-1} r^{n+m-k}, \quad 1\le k\le n.
\end{equation}

\subsection{Notation}
Throughout this work we use the notations $\N=\{1,2,3,\dots\}$ and $\N_0=\{0,1,2,\dots\}$. Furthermore, we denote with $i$ the imaginary unit, $i^2=-1$. 
We use the abbreviations $\fallfak{x}{s} :=
x (x-1) \cdots (x-s+1)$ for the falling factorials.
Further we denote with $\stir{n}{k}$ the unsigned Stirling numbers of the first kind and with $\Stir{n}{k}$ the
Stirling numbers of the second kind. We use the notation $X_{n} \xrightarrow{(d)} X$  for the  weak convergence, i.~e. the convergence
in distribution, of the sequence of random variables $X_{n}$ to a random variable $X$.

\subsection{Overview}
We derive explicit formulas for the probability mass functions of the two urn models I and II for two given sequences $(A,B)$ of positive numbers $A=(\alpha_n)_{n\in\N}$ and $B=(\beta_m)_{m\in\N}$
assuming that both sequences $(A,B)$ satisfy the conditions $\alpha_j \neq \alpha_{\ell}$ and $\beta_j \neq \beta_{\ell}$, $1\le j<\ell<\infty$. 
Our explicit results concerning the distribution of $X_{n,m}$ and the resulting expressions for the P\'olya-Eggenberger urn models generalize the corresponding results of \cite{FlaDumPuy2006,HwaKuPa2007}. Furthermore, we also study higher dimensional urn models with $r\ge 2$ types of balls. Our approach consists 
of constructing a family of functions $(f_{n,m}(t))_{n,m\in\N}$, which can be related to the probability generating function
of the random variables $X_{n,m}$. It is based on the work of Stadje~\cite{Stad1998}, see also Kingman and Volkov~\cite{KinVol2003}.
Stadje did derive an integral representation for the probabilites $\P\{X_{n,m}=k\}$ for urn model II. In this work we revisit Stadje's approach and extend 
it to urn model I in order to obtain integral representations for both urn models, which in turn will be used to derive 
the explicit expressions for the probability mass function $\P\{X_{n,m}=k\}$.

\smallskip

Furthermore, concerning the P\'olya-Eggenberger versions of urn models I and II we will derive various expressions 
for the moments of the random variables $X_{n,m}$ (or more precisely $Y_{an,dm}$ and $Y_{cn,bm}$).
Later, we extend urn models I and II to higher dimensions by considering $r\ge 2$ different types of balls.
Moreover, we will prove that both urn models are dual to each other in a certain sense, see Section~\ref{GENsecdual}, which also extends to the higher dimensional urn models. Finally, in Section~\ref{GENSAMPthetasec} we study urn model I with weight sequences $(\alpha_n)_{n\in\N}=(n)_{n\in\N}$ and $(\beta_m)_{m\in\N}=(m^2)_{m\in\N}$ and show that interesting limiting distributions arise, curiously involving the Jacobi Theta function $\Theta(q)=\sum_{n\in\Z}(-1)^nq^{n^2}$.

\section{Main results}
\subsection{Results for urn model I}

\begin{theorem}
\label{GENthe1}
The probability mass function of the random variable $X_{n,m}$, counting the number of remaining white balls when all black balls
have been drawn in urn model I with weight sequences $A=(\alpha_n)_{n\in\N}$, $B=(\beta_m)_{m\in\N}$ (Sampling with replacement with general weights)), is for $n,m\ge 1$ and $n\ge k\ge 0$ given by the explicit formula 
\begin{equation*}
\begin{split}
\P\{X_{n,m}=k\}&= \Big(\prod_{h=1}^{m}\beta_h\Big)\Big(\prod_{h=k+1}^{n}\alpha_h\Big)\sum_{\ell=1}^{m}
\frac{1}{\Big(\prod_{j=k}^{n}(\alpha_j+\beta_{\ell})\Big)\Big(\prod_{\substack{i=1\\i\neq \ell}}^{m}(\beta_{i}-\beta_{\ell})\Big)}\\
&= \Big(\prod_{h=1}^{m}\beta_h\Big)\Big(\prod_{h=k+1}^{n}\alpha_h\Big)\sum_{\ell=k}^{n}
\frac{1}{\Big(\prod_{\substack{j=k\\j\neq \ell}}^{n}(\alpha_j-\alpha_{\ell})\Big)\Big(\prod_{i=1}^{m}(\beta_{i}+\alpha_{\ell})\Big)},
\end{split}
\end{equation*}
assuming that $\alpha_j \neq \alpha_{\ell}$ and $\beta_j \neq \beta_{\ell}$, $1\le j<\ell<\infty$, and that $\alpha_0=0$.
\end{theorem}

Adapting the weights sequences $(A,B)$ to the special case $A=(a\cdot n)_{n\in\N}$ and $B=(d\cdot m)_{m\in\N}$, with $a,b\in\N$, corresponding to the 
generalized sampling P\'olya-Eggenberger urn model with ball transition matrix $M = \bigl(\begin{smallmatrix} -a & 0 \\ 0 & -d\end{smallmatrix}\bigr)$, 
leads to the following result.

\begin{coroll}
The probability mass function of the random variable $Y_{an,dm}$, counting the number of remaining white balls when all black balls have been removed
in the generalized sampling without replacement P\'olya-Eggenberger urn model with ball transition matrix $M = \bigl(\begin{smallmatrix} -a & 0 \\ 0 & -d\end{smallmatrix}\bigr)$,
starting with $a\cdot n$ white and $d\cdot n$ black balls,
is given by 
\begin{equation*}
 \P\{Y_{an,dm}=ak\}= \sum_{\ell=1}^{m}(-1)^{\ell-1}\frac{\binom{m}{l}\binom{k-1+\frac{\ell d}{a}}{k}}{\binom{n+\frac{\ell d}{a}}{n}}
 =\sum_{\ell=k}^{n}(-1)^{\ell-k}\frac{\binom{n}{\ell}\binom{\ell}{k}}{\binom{m+\frac{\ell a}{d}}{m}}
   ,\quad 0\le k\le n.
\end{equation*}
The moments of the normalized random variable $\hat{Y}_{an,dm}=Y_{an,dm}/a$ are given by the following expressions.
\begin{equation*}
\E\Big(\fallfak{\hat{Y}_{an,dm}}{s}\Big)=\frac{\fallfak{n}{s}}{\binom{m+\frac{as}{d}}{m}},\quad
\E\big(\hat{Y}_{an,dm}^s\Big)=\sum_{j=0}^{s}\Stir{s}{j}\frac{\fallfak{n}{j}}{\binom{m+\frac{aj}{d}}{m}},
\end{equation*}
where $\Stir{n}{k}$ denote the Stirling number of the second kind, and $\fallfak{x}{s}=x(x-1)(x-2)\dots(x-(s-1))$ the falling factorials.
\end{coroll}

\subsection{Results for urn model II}

\begin{theorem}
\label{GENthe2}
The probability mass function of the random variable $X_{n,m}$, counting the number of remaining white balls when all black balls
have been drawn in urn model II with weight sequences $A=(\alpha_n)_{n\in\N}$, $B=(\beta_m)_{m\in\N}$ (OK Corral urn model with general weights), is for $n,m\ge 1$ and $n\ge k\ge 1$ given by the explicit formula 
\begin{equation*}
\begin{split}
\P\{X_{n,m}=k\}&= \alpha_k \sum_{j=k}^{n} \frac{\alpha_j^{m+n-k-1}}{\Big(\prod_{\substack{\ell=k\\\ell\neq j}}^{n}(\alpha_j-\alpha_{\ell})\Big)
\Big(\prod_{h=1}^{m}(\alpha_j+\beta_h)\Big)}\\
&=\alpha_k \sum_{\ell=1}^{m} \frac{\beta_{\ell}^{n+m-1-k}}{\Big(\prod_{j=k}^{n}(\beta_{\ell}+\alpha_{j})\Big)
\Big(\prod_{\substack{h=1\\h \neq \ell}}^{m}(\beta_{\ell}-\beta_h)\Big)},
\end{split}
\end{equation*}
assuming that $\alpha_j \neq \alpha_{\ell}$ and $\beta_j \neq \beta_{\ell}$, $1\le j<\ell<\infty$. 
For $n,m\ge 1$ and $k=0$ we obtain
\begin{equation*}
\begin{split}
\P\{X_{n,m}=0\}&= 1- \sum_{j=1}^{n}\frac{\alpha_{j}^{n+m-1}}{\Big(\prod_{\substack{\ell=1\\\ell\neq j}}^{n}(\alpha_j-\alpha_{\ell})\Big)
\Big(\prod_{h=1}^{m}(\alpha_j+\beta_h)\Big)}\\
&=\sum_{\ell=1}^{m} \frac{\beta_{\ell}^{n+m-1}}{\Big(\prod_{j=1}^{n}(\beta_{\ell}+\alpha_{j})\Big)
\Big(\prod_{\substack{h=1\\h \neq \ell}}^{m}(\beta_{\ell}-\beta_h)\Big)}.
\end{split}
\end{equation*}
\end{theorem}

Adapting the weights sequences $(A,B)$ to the special case $A=(c\cdot n)_{n\in\N}$, $B=(b\cdot m)_{m\in\N}$, with $b,c\in\N$ corresponding to the 
generalized OK-Corral P\'olya-Eggenberger urn model with ball transition matrix $M = \bigl(\begin{smallmatrix} 0 & -b \\ -c & 0\end{smallmatrix}\bigr)$,
we obtain the following results.

\begin{coroll}
The probability mass function of the random variable $Y_{cn,bm}$, counting the number of remaining white balls when all black balls have been removed
in the generalized OK-Corral P\'olya-Eggenberger urn model with ball transition matrix $M = \bigl(\begin{smallmatrix} 0 & -b \\ -c & 0\end{smallmatrix}\bigr)$
starting with $c\cdot n$ white and $b\cdot n$ black balls, is given by 
\begin{align*}
   \P\{Y_{cn,bm}=ck\}&
   = \frac{k}{(n-k+1)!(m-1)!} \frac{b^{n-k}}{c^{n-k}} \sum_{l=1}^{m} (-1)^{m-l}
    \frac{ \binom{m-1}{l-1}}{ \binom{n+\frac{b}{c}l}{n-k+1}} l^{n+m-1-k},\\
  &=  \frac{k}{(n-k)!m!} \frac{c^{m}}{b^{m}} 
    \sum_{l=0}^{n} (-1)^{n-l} \frac{\binom{n-k}{l-k}}{\binom{m+\frac{cl}b}{m}}l^{m+n-1-k}  
    \quad
    \text{for}\quad 1\le k\le n,\\
   \P\{Y_{cn,bm}=0\}&
   = \frac{1}{n!(m-1)!} \frac{b^{n}}{c^{n}} \sum_{l=1}^{m} (-1)^{m-l}
    \frac{ \binom{m-1}{l-1}}{ \binom{n+\frac{b}{c}l}{n}} l^{n+m-1},\quad
    \text{for}\quad k=0.
\end{align*}
\end{coroll}

Using the explicit expressions for the probabilities and an important result of Puyhaubert~\cite{PuyPhD} and Flajolet~\cite{FlaTalk}, originally developed for the analyis of the case $b=c$, we can give a precise description of the moments.
\begin{coroll}
The moments of the random variable $\hat{Y}_{cn,bm}=Y_{cn,bm}/c$, where $Y_{cn,bm}$ counts the number of remaining white balls when all black balls have been removed
in the generalized OK-Corral P\'olya-Eggenberger urn model with ball transition matrix $M = \bigl(\begin{smallmatrix} 0 & -b \\ -c & 0\end{smallmatrix}\bigr)$
starting with $c\cdot n$ white and $b\cdot n$ black balls, are given by the exact formula 
\begin{equation*}
\E(\hat{Y}_{cn,bm}^s)= \frac{1}{(n+m)!} \frac{c^{m}}{b^{m}} 
    \sum_{\ell=0}^{n} (-1)^{n-\ell}
    \frac{ \binom{n+m}{n-\ell}\binom{m+l}{\ell} }{\binom{m+\frac{c\ell}b}{m}}\ell^{m+n-1}(f_{s+1}(\ell)Q(\ell)+g_{s+1}(\ell)), 
\end{equation*}
where $f_n(u)=n![z^n]F(z,u)$, $g_n(u)=n![z^n]G(z,u)$, with
\begin{equation*}
F(z,u)= e^{u\big(e^{-z}+z-1\big)},\qquad G(z,u)=e^{u\big(e^{-z}+z-1\big)}\int_0^z u e^{-t}e^{-u\big(e^{-t}+t-1\big)}dt,
\end{equation*}
and $Q(n)=\sum_{i=0}^n\frac{\fallfak{n}{i}}{n^i}$ denoting Ramanujan's Q-function.
Moreover, there exists monic polynomials $(M_s)_{s\in\N}$ of degree $2s$ such that 
\begin{equation*}
\E(M_s(\hat{Y}_{cn,bm}))= \frac{s!2^s}{(n+m)!} \frac{c^{m}}{b^{m}} 
    \sum_{\ell=0}^{n} (-1)^{n-\ell}
    \frac{ \binom{n+m}{n-\ell}\binom{m+\ell}{\ell} }{\binom{m+\frac{c\ell}b}{m}}\ell^{m+n+s}=\frac{s!2^s}{n!m!} \frac{c^{m}}{b^{m}} 
    \sum_{\ell=0}^{n} (-1)^{n-\ell}
    \frac{ \binom{n}{\ell}}{\binom{m+\frac{c\ell}b}{m}}\ell^{m+n+s},
\end{equation*}
where the coefficients $m_{i,s}$, with $1\le i \le 2s$, of $M_s$ are specified by the triangular system of linear equations
\begin{equation*}
\sum_{i=1}^{2s}m_{i,s}f_{i+1}(X)=0,\qquad \sum_{i=1}^{2s}m_{i,s}g_{i+1}(X)=s!2^s X^{s+1}.
\end{equation*}
\end{coroll}

\section{Deriving the explicit results: Urn model I}
By definition of urn model I the probabilities $\P\{X_{n,m}=k\}$ satisfy the following system of recurrence relations.
\begin{equation}
\label{GENSAMPrec1}
\begin{split}
\P\{X_{n,m}=k\}&=\frac{\alpha_n}{\alpha_n+\beta_m}\P\{X_{n-1,m}=k\} + \frac{\beta_m}{\alpha_n+\beta_m}\P\{X_{n,m-1}=k\},
\quad m \ge 1, \quad n\ge 1,\quad n\ge k\ge 0\\
\P\{X_{n,0}=k\}&=\delta_{k,n},\quad n\ge 1,\,k\ge 0, \quad\text{and}\quad \P\{X_{0,m}=k\}=\delta_{k,0},\quad m\ge 1,\,k\ge 0.
\end{split}
\end{equation}
In order to prove the result above we first derive a class of formal solution of the system of recurrence relations~\eqref{GENSAMPrec1}.
Then, out of the class of formal solutions we determine the proper solution by adapting to the initial conditions $\P\{X_{n,0}=n\}=1$ for $n\ge 0$.
We will present the proof using a sequence of lemmata.
\begin{lemma}
\label{GENSAMPlem1}
The quantities $f_{n,m,k}(t)\in\C$, defined as 
\begin{equation*}
f_{n,m,k}(t)=\frac{\xi_k(t) \Big(\prod_{j=1}^{k-1}(1-\frac{t}{\alpha_j})\Big)}{\Big(\prod_{j=1}^{n}(1-\frac{t}{\alpha_j})\Big)\Big(\prod_{\ell=1}^{m}(1+\frac{t}{\beta_{\ell}})\Big)},
\end{equation*}
with parameter $\xi_k(t)\in\C$, satisfy for each fixed $t\in i\cdot\R$ and $n,m\ge 1$ and $k\ge 1$ a recurrence relation of the type~\eqref{GENSAMPrec1}. 
Moreover, $f_{n,m,k}(t)$ is a meromorphic function in $\C$, with poles at the values $\{-\beta_{\ell}\mid 1\le \ell\le m\}$ and $\{\alpha_j\mid k\le j\le n\}$.
\end{lemma}
\begin{proof}
We observe that for arbitrary $n,m,k\ge 1$
\begin{equation*}
f_{n-1,m,k}(t)=f_{n,m,k}(t)\Big(1-\frac{t}{\alpha_n}\Big),\quad f_{n,m-1,k}(t)=f_{n,m,k}(t)\Big(1+\frac{t}{\beta_m}\Big).
\end{equation*}
Consequently, 
\begin{equation*}
\frac{\alpha_n}{\alpha_n+\beta_m} f_{n-1,m,k}(t) + \frac{\beta_m}{\alpha_n+\beta_m}f_{n,m-1,k}(t)
= \frac{f_{n,m,k}(t)}{\alpha_n+\beta_m}\Big(\alpha_n-t + \beta_m +t\Big)=f_{n,m,k}(t),
\end{equation*}
which proves that $f_{n,m,k}(t)$ satisfies a recurrence relation of the type~\eqref{GENSAMPrec1}.
\end{proof}
Note that we exclude the case $k=0$, since we assume that $\alpha_0=0$. 
Since all of the quantities $f_{n,m,k}(t)\in\C$ are solutions of a recurrence relation of the type~\eqref{GENSAMPrec1}, 
we can derive another family of solutions.

\begin{lemma}
\label{GENSAMPlem2}
The values $F_{n,m,k}=\frac{1}{2\pi i}\int_{-i\infty}^{i\infty}f_{n,m,k}(t)dt $ are for $n,m,k\in\N$ formal solutions of the system of recurrence relations~\eqref{GENSAMPrec1}, with parameter $\xi_k\in\C$ assumed to be a constant independent of $t$. For $n\ge k\ge 1$ and $m\ge 1$ we can obtain $F_{n,m,k}$ by two different contours of integration,
\begin{equation*}
\begin{split}
F_{n,m,k}&=\frac{1}{2\pi i}\int_{\mathcal{C}_1}f_{n,m,k}(t)dt=\frac{1}{2\pi i}\int_{\mathcal{C}_2}f_{n,m,k}(t)dt,\quad \text{with}\quad \mathcal{C}_1=\gamma_{1}\cup\gamma_{2},
\quad \mathcal{C}_2=\gamma_{1}\cup\gamma_{3},
\end{split}
\end{equation*}
where $\gamma_1(z)=-iz$, $-R\le z\le R$, $\gamma_2(z)=Re^{i\varphi}$, $\pi/2\le \varphi \le 3\pi/2$, $\gamma_3(z)=R e^{i(\pi-\varphi)}$, $\pi/2\le \varphi \le 3\pi/2$,
such that $R>\max_{1\le \ell \le m, k\le j\le n}\{\beta_{\ell},\alpha_j\}$.
\end{lemma}

\begin{proof}
By definition of $f_{n,m,k}(t)$ the first stated asseration easily follows. By Lemma~\ref{GENSAMPlem1} $f_{n,m,k}(t)$ is a meromorphic function in $\C$, with simple poles at the values $\{-\beta_{\ell}\mid 1\le \ell\le m\}$ and $\{\alpha_j\mid k\le j\le n\}$. By taking the limit $R\to\infty$ in the two different contours of integration
we observe that the contributions of $\gamma_2$ and $\gamma_3$ are zero in the limit, and only the contribution of $\gamma_1$ 
remains. Hence, 
\begin{equation*}
\begin{split}
\frac{1}{2\pi i}\int_{-i\infty}^{i\infty}f_{n,m,k}(t)dt &=\frac{1}{2\pi i}\int_{\mathcal{C}_1}f_{n,m,k}(t)dt=\frac{1}{2\pi i}\int_{\mathcal{C}_2}f_{n,m,k}(t)dt.
\end{split}
\end{equation*}
\end{proof} 
Next we will derive an alternative representations for $F_{n,m,k}$ using the formal residue calculus.
\begin{lemma}
\label{GENSAMPlem3}
Assume that the weight sequence $A=(\alpha_n)_{n\in\N}$ satisfies $\alpha_i\neq \alpha_j$ for all $i,j\in\N$ with $i\neq j$.
Then the value $F_{n,m,k}$ admits for $n\ge k\ge 1$ and $m,\ge 1$ the following representation.
\begin{equation*}
 F_{n,m,k}=\xi_k\Big(\prod_{h=1}^{m}\beta_h\Big)\Big(\prod_{h=k}^{n}\alpha_h\Big)\sum_{\ell=k}^{n}
\frac{1}{\Big(\prod_{\substack{j=k\\j\neq \ell}}^{n}(\alpha_j-\alpha_{\ell})\Big)\Big(\prod_{i=1}^{m}(\beta_{i}+\alpha_{\ell})\Big)}.
\end{equation*}
Assume that the weight sequence $B=(\beta_m)_{m\in\N}$ satisfies $\beta_i\neq \beta_j$ for all $i,j\in\N$ with $i\neq j$.
Then the value $F_{n,m,k}$ admits for $n\ge k\ge 1$ and $m,\ge 1$ the following representation.
\begin{equation*}
F_{n,m,k}=\xi_k\Big(\prod_{h=1}^{m}\beta_h\Big)\Big(\prod_{h=k}^{n}\alpha_h\Big)\sum_{\ell=1}^{m}
\frac{1}{\Big(\prod_{j=k}^{n}(\alpha_j+\beta_{\ell})\Big)\Big(\prod_{\substack{i=1\\i\neq \ell}}^{m}(\beta_{i}-\beta_{\ell})\Big)}.
\end{equation*}
\end{lemma}

\begin{proof}
By definition of $F_{n,m,k}$ and Lemma~\ref{GENSAMPlem2} we have 
\begin{equation*}
\begin{split}
F_{n,m,k}&=\frac{\xi_k}{2\pi i}\int_{\mathcal{C}_2}f_{n,m,k}(t)dt=\frac{\xi_k}{2\pi i}\int_{\mathcal{C}_2}\frac{dt}{\Big(\prod_{j=k}^{n}(1-\frac{t}{\alpha_j})\Big)\Big(\prod_{\ell=1}^{m}(1+\frac{t}{\beta_{\ell}})\Big)}\\
&=\Big(\prod_{h=1}^{m}\beta_h\Big)\Big(\prod_{h=k}^{n}\alpha_h\Big)\frac{\xi_k}{2\pi i}\int_{\mathcal{C}_2}\frac{dt}{\Big(\prod_{j=k}^{n}(\alpha_j-t)\Big)\Big(\prod_{\ell=1}^{m}(\beta_{\ell}+t)\Big)},
\end{split}
\end{equation*}
By the residue theorem with respect to the poles $-\alpha_{j}$, $k\le \ell \le j$, we obtain the representation
\begin{equation*}
F_{n,m,k}= \sum_{h=k}^{n}\text{Res}(f_{n,m,k},\alpha_{h}) = \xi_k\Big(\prod_{h=1}^{m}\beta_h\Big)\Big(\prod_{h=k}^{n}\alpha_h\Big)\sum_{\ell=k}^{n}
\frac{1}{\Big(\prod_{\substack{j=k\\j\neq \ell}}^{n}(\alpha_j-\alpha_{\ell})\Big)\Big(\prod_{i=1}^{m}(\beta_{i}+\alpha_{\ell})\Big)},
\end{equation*}
by the assumption on the weight sequence $A$. A similar argument leads to the second representation for $F_{n,m,k}$.
\end{proof}
Note that one can alternatively convert between the two expression using the partial fraction identity
\begin{equation}
\label{GENSAMPparfrac}
\frac{1}{\prod_{j=k}^{n}(\alpha_j+x)}= \sum_{h=k}^n\frac{1}{(x+\alpha_h)\prod_{\substack{j=k\\j\neq h}}^{n}(\alpha_j-\alpha_h)}.
\end{equation}
For example, applying the partial fraction identity to the second expression gives
\begin{equation*}
F_{n,m,k}= \xi_k\Big(\prod_{h=1}^{m}\beta_h\Big)\Big(\prod_{h=k}^{n}\alpha_h\Big)\sum_{j=k}^{n}\frac{1}{\Big(\prod_{\substack{j=k\\j\neq h}}^{n}(\alpha_j-\alpha_{h})\Big)}
\sum_{\ell=1}^{m} \frac{1}{(\beta_{\ell}+\alpha_h)\Big(\prod_{\substack{i=1\\i\neq \ell}}^{m}(\beta_{i}-\beta_{\ell})\Big)}.
\end{equation*}
Another application of the partial fraction identity provides the first expression.

\smallskip

We still have to adapt the parameter $\xi_k$ in the representation of $F_{n,m,k}$ to the initial conditions $\P\{X_{n,0}=k\}=\delta_{n,k}$ for $n,k\ge 0$,
and the check the initial conditions. First we will consider the special value $F_{n,m,m}$ in order to choose the parameter $\xi_k$ in the representation of $F_{n,m,k}$ based on the explicit expression. Then we will check if the initial conditions are fulfilled. Let $k=n$. We have
\begin{equation*}
F_{n,m,n}= \xi_n\Big(\prod_{h=1}^{m}\beta_h\Big)\alpha_n\sum_{\ell=1}^{m}
\frac{1}{(\alpha_n+\beta_{\ell})\Big(\prod_{\substack{i=1\\i\neq \ell}}^{m}(\beta_{i}-\beta_{\ell})\Big)}.
\end{equation*}
Application of~\eqref{GENSAMPparfrac} with respect to the case $k=n$ gives
\begin{equation*}
F_{n,m,n}= \frac{\xi_n\Big(\prod_{h=1}^{m}\beta_h\Big)\alpha_n}{\prod_{h=1}^{m}(\beta_h+\alpha_n)}.
\end{equation*}
Setting $m=0$ we obtain $F_{n,0,n}=\xi_n\alpha_n$. In order to adapt to the special cases of the initial conditions $\P\{X_{n,0}=n\}=1$ for $n\ge 1$
we choose $\xi_k=1/\alpha_k$ for $n\ge 1$ and $k\ge 1$. Then we extend the definition of $f_{n,m,k}(t)$ and $F_{n,m,k}$ to cope with the additional boundary cases $n=0$, $m=0$ and $k=0$. We obtain the desired representation by a slight modification of the definition of $f_{n,m,k}(t)$ in the case of $k=0$.
\begin{lemma}
\label{GENSAMPlem5}
For $n,m\ge 1$ and $k\ge 0$ the values $\P\{X_{n,m}=k\}=\frac{1}{2\pi i}\int_{\mathcal{C}_3}\hat{f}_{n,m,k}(t)dt$, with 
\begin{equation*}
\hat{f}_{n,m,k}(t)=\begin{cases}
\frac{\frac{1}{\alpha_k}\Big(\prod_{j=1}^{k-1}(1-\frac{t}{\alpha_j})\Big)}{\Big(\prod_{j=1}^{n}(1-\frac{t}{\alpha_j})\Big)\Big(\prod_{\ell=1}^{m}(1+\frac{t}{\beta_{\ell}})\Big)},\quad \text{for}\quad
k\ge 1,\\[0.5cm]
\frac{1}{(-t)\Big(\prod_{j=1}^{n}(1-\frac{t}{\alpha_j})\Big)\Big(\prod_{\ell=1}^{m}(1+\frac{t}{\beta_{\ell}})\Big)},\quad \text{for}\quad k=0,
\end{cases}
\end{equation*}
where $\mathcal{C}_3=\gamma_{1}'\cup\gamma_{2}'\cup\gamma_{3}'\cup\gamma_{4}$, $\gamma_1'(z)=-iz$, $-R\le z\le \varepsilon$, $\gamma_2'(z)=\varepsilon e^{i(3\pi/2-\varphi)}$, $0\le \varphi \le \pi$, $\gamma_3'(z)=iz$, $\varepsilon\le z\le R$ and $\gamma_4'(z)=R e^{i(\pi-\varphi)}$, $\pi/2\le \varphi \le 3\pi/2$, such that $R>\max_{1\le \ell \le m, k\le j\le n}\{\beta_{\ell},\alpha_j\}$ and $\varepsilon<\min_{1\le \ell \le m, k\le j\le n}\{\beta_{\ell},\alpha_j\}$,
are the proper solution of the system of recurrence relations~\eqref{GENSAMPrec1} fulfilling the stated initial conditions involding the boundary cases $n=0$ and $m=0$.
\end{lemma}

\begin{proof}
First we check that the initial conditions $\P\{X_{n,0}=k\}=\delta_{k,n}$ for $n\ge 1$ and $k\ge 0$. 
For $k=n$ we can use our previous computation which lead to the choice $\xi_k=1/\alpha_k$, and
the stated result easily follows using~\eqref{GENSAMPparfrac}. For $k\neq n$ we use residue theorem and then the partial fraction decomposition
\begin{equation*}
\frac{-x}{(x+\alpha_k)(x+\alpha_{k+1})\dots (x+\alpha_n)}= \sum_{\ell=k}^{n} \frac{\alpha_{\ell}}{(x+\alpha_{\ell})\Big(\prod_{\substack{i=k\\i\neq \ell}}^{n}(\alpha_{i}-\alpha_{\ell})\Big)}.
\end{equation*}
Taking the limit $x\to 0$ in the identity above provides the needed result. Finally, we check the initial conditions $\P\{X_{0,m}=k\}=\delta_{k,0}$ for $m\in\N$ and $k\in\N_0$.
For $k=0$ we immediately obtain $\P\{X_{0,m}=0\}=1$, due to the additional simple pole at $t=0$. Concerning the remaining cases $k\ge 1$ we rely on the integral representation of the probabilities $\frac{1}{2\pi i}\int_{\mathcal{C}_3}\hat{f}_{n,m,k}(t)dt$. By Lemma~\ref{GENSAMPlem1} $\hat{f}_{n,m,k}(t)$ has no singularities with negative real part (more precisely, it has removeable singularities at $\{\alpha_j\mid 1\le j\le n\}$). 
Hence, by the Cauchy integral theorem $\P\{X_{0,m}=k\}=\frac{1}{2\pi i}\int_{\mathcal{C}_3}\hat{f}_{0,m,k}(t)dt=0$ for $k\ge 1$ and $n=0$.
\end{proof}

\begin{remark}
It seems difficult prove that the initial conditions $\P\{X_{0,m}=k\}=\delta_{k,0}$ for $m\in\N$ and $k\in\N_0$ are satisfied 
relying on the other integral representation, or solely on the two explicit expressions valid for $n\ge k\ge 1$ and $m\ge 1$.
\end{remark}

\begin{proof}[Proof of Theorem~1]
By Lemma~\ref{GENSAMPlem3} we obtain for $m,n\ge 1$ and $k\ge 1$ the stated explicit expressions.
Our subsequent calculations 
show that stated expressions are also valid for $k=0$. By Lemma~\ref{GENSAMPlem1} and Lemma~\ref{GENSAMPlem5} the recurrence relations and the initial conditions are are fulfilled.
It remains to prove that the stated formulas are also valid in the case $k=0$. 
By the residue theorem we obtain 
\begin{equation*}
\begin{split}
\P\{X_{n,m}=0\}&=\frac{1}{2\pi i}\int_{\mathcal{C}_3}\hat{f}_{n,m,0}(t)dt
=1-\sum_{\ell=1}^{n} \frac{\Big(\prod_{h=1}^{m}\beta_h\Big)\Big(\prod_{h=1}^{n}\alpha_h\Big)}{\alpha_{\ell}\Big(\prod_{\substack{j=1\\j\neq \ell}}^{n}(\alpha_j-\alpha_{\ell})\Big)\Big(\prod_{i=1}^{m}(\beta_{i}+\alpha_{\ell})\Big)}\\
&=1+\sum_{\ell=1}^{n} \frac{\Big(\prod_{h=1}^{m}\beta_h\Big)\Big(\prod_{h=1}^{n}\alpha_h\Big)}{\Big(\prod_{\substack{j=0\\j\neq \ell}}^{n}(\alpha_j-\alpha_{\ell})\Big)\Big(\prod_{i=1}^{m}(\beta_{i}+\alpha_{\ell})\Big)}\\
&=\sum_{\ell=0}^{n} \frac{\Big(\prod_{h=1}^{m}\beta_h\Big)\Big(\prod_{h=1}^{n}\alpha_h\Big)}{\Big(\prod_{\substack{j=0\\j\neq \ell}}^{n}(\alpha_j-\alpha_{\ell})\Big)\Big(\prod_{i=1}^{m}(\beta_{i}+\alpha_{\ell})\Big)},
\end{split}
\end{equation*}
subject to $\alpha_0=0$, which proves the first part of the stated result for $k=0$. The second formula follows easily using the partial fraction decomposition~\eqref{GENSAMPparfrac}.
Alternatively, one may change the contour of integration $\mathcal{C}_3$ and collect the residues at the poles $-\beta_1,\dots,-\beta_m$.
\end{proof}

\begin{proof}[Proof of Corollary 1]
In order to obtain the results for the P\'olya-Eggenberger generalized sampling urn model with ball transition matrix $M = \bigl(\begin{smallmatrix} -a & 0 \\ 0 & -d\end{smallmatrix}\bigr)$, 
we simple adapt the weights sequences $(A,B)$ to the special case $A=(a\cdot n)_{n\in\N}$ and $B=(d\cdot m)_{m\in\N}$, with $a,b\in\N$.
Concerning the moments we use the well know, see e.g.~\cite{GraKnuPat1994}, basic identities
\begin{equation}
\label{GENSAMPeqn8}
\sum_{k=0}^{n}\binom{x+k}{k}=\binom{x+n+1}{n},\quad \sum_{k=0}^{n}\binom{n}{k}\frac{(-1)^k}{x+k}=\frac{1}{x\binom{x+n}{n}}.
\end{equation}
Let $\hat{Y}_{an,dm}$ denote the scaled random variable $Y_{an,dm}/a$.
The factorial moments of $\hat{Y}_{an,dm}$ are derived as follows.
\begin{equation*}
\E(\fallfak{\hat{Y}_{an,dm}}{s})=\sum_{k=0}^{n}\fallfak{k}{s}\P\{X_{an,dm}=ka\}
= \sum_{k=s}^{n}\sum_{\ell=1}^{m}(-1)^{\ell-1}\frac{\binom{m}{l}\binom{k-1+\frac{\ell d}{a}}{k}\fallfak{k}{s}}{\binom{n+\frac{\ell d}{a}}{n}}.
\end{equation*}
By interchanging summation and using the first identity stated in \eqref{GENSAMPeqn8} we get further
\begin{equation*}
\begin{split}
\E(\fallfak{\hat{Y}_{an,dm}}{s})&= \sum_{\ell=1}^{m}(-1)^{\ell-1}\frac{\binom{m}{l}\binom{n+s+\frac{\ell d}{a}}{n-s}\fallfak{(s-1+\frac{\ell d}{a})}{s}}{\binom{n+\frac{\ell d}{a}}{n}}
= \fallfak{n}{s}\sum_{\ell=1}^{m}(-1)^{\ell-1}\frac{\binom{m}{l}}{s+\frac{\ell d}{a}}.
\end{split}
\end{equation*}
The application of the second identity stated in \eqref{GENSAMPeqn8} and the fact that 
\begin{equation}
\label{MOMREL}
  \E(\hat{Y}_{an,dm}^{s})  = \sum_{k=1}^{s} \Stir{s}{k} \E(\fallfak{\hat{Y}_{an,dm}}{k}),
\end{equation}
where $\Stir{n}{k}$ denotes the Stirling numbers of the second kind lead obtain the following results.
\begin{equation*}
\E(\fallfak{\hat{Y}_{an,dm}}{s})=\frac{\fallfak{n}{s}}{\binom{m+\frac{as}{d}}{m}},\quad \E(\hat{Y}_{an,dm}^s)=\sum_{j=0}^{s}\Stir{s}{j}\frac{\fallfak{n}{j}}{\binom{m+\frac{aj}{d}}{m}}.
\end{equation*}
\end{proof}

\section{Deriving the explicit results: Urn model II}
The derivation of the explicit results for urn model II is similar to are previous approach to urn model I; therefore we will be more brief.
By definition the probabilities $\P\{X_{n,m}=k\}$ satisfy the following system of recurrence relations.
\begin{equation}
\label{GENSAMPrec2}
\begin{split}
\P\{X_{n,m}=k\}&=\frac{\beta_m}{\alpha_n+\beta_m}\P\{X_{n-1,m}=k\} + \frac{\alpha_n}{\alpha_n+\beta_m}\P\{X_{n,m-1}=k\},
\quad m \ge 1, \quad n\ge 1,\quad n\ge k\ge 0\\
\P\{X_{n,0}=k\}&=\delta_{k,n}, \quad\text{and}\quad \P\{X_{0,m}=k\}=\delta_{k,0},\quad m\ge 1.
\end{split}
\end{equation}
As before, we first derive a class of formal solution of the system of recurrence relations~\eqref{GENSAMPrec2}.
Then, out of the class of formal solutions we determine the proper solution by adapting to the initial conditions $\P\{X_{n,0}=n\}=1$ for $n\ge 0$.
\begin{lemma}
\label{GENSAMPlemOK1}
The quantities $g_{n,m,k}(t)\in\C$, defined as 
\begin{equation*}
g_{n,m,k}(t)=\frac{\xi_k \Big(\prod_{j=1}^{k-1}(1-\alpha_j t)\Big)}{\Big(\prod_{j=1}^{n}(1- \alpha_j t)\Big)\Big(\prod_{\ell=1}^{m}(1+ \beta_{\ell} t)\Big)},
\end{equation*}
with parameter $\xi_k\in\C$, satisfy for each fixed $t\in i\cdot\R$ and $n,m\ge 1$ and $k\ge 1$ a recurrence relation of the type~\eqref{GENSAMPrec2}. 
Moreover, $g_{n,m,k}(t)$ is a meromorphic function in $\C$, with poles at the values $\{-1/\beta_{\ell}\mid 1\le \ell\le m\}$ and $\{1/\alpha_j\mid k\le j\le n\}$.
\end{lemma}
\begin{proof}
We observe that for arbitrary $n,m,k\ge 1$
\begin{equation*}
g_{n-1,m,k}(t)=g_{n,m,k}(t)\Big(1- \alpha_n t \Big),\quad g_{n,m-1,k}(t)=g_{n,m,k}(t)\Big(1+ \beta_m t \Big).
\end{equation*}
Consequently, 
\begin{equation*}
\frac{\beta_m}{\alpha_n+\beta_m} g_{n-1,m,k}(t) + \frac{\alpha_n}{\alpha_n+\beta_m}g_{n,m-1,k}(t)
= \frac{g_{n,m,k}(t)}{\alpha_n+\beta_m}\Big(\beta_m- \beta_m \alpha_nt + \alpha_n +\alpha_n \beta_m t\Big)=g_{n,m,k}(t),
\end{equation*}
which proves that $g_{n,m,k}(t)$ satisfies a recurrence relation of the type~\eqref{GENSAMPrec1}.
\end{proof}
Next we derive another family of solutions.

\begin{lemma}
\label{GENSAMPlemOK2}
The value $G_{n,m,k}=\frac{1}{2\pi i}\int_{-i\infty}^{i\infty}g_{n,m,k}(t)dt $ is for $n,m,k\in\N$ a formal solution of the system of recurrence relations~\eqref{GENSAMPrec2},
with parameter $\xi_k\in\C$ assumed to be a constant independent of $t$. For $n\ge k\ge 1$ and $m\ge 1$ we can obtain $G_{n,m,k}$ by two different contours of integration,
\begin{equation*}
\begin{split}
G_{n,m,k}&=\frac{1}{2\pi i}\int_{\mathcal{C}_1}g_{n,m,k}(t)dt=\frac{1}{2\pi i}\int_{\mathcal{C}_2}g_{n,m,k}(t)dt,\quad \text{with}\quad \mathcal{C}_1=\gamma_{1}\cup\gamma_{2},
\quad \mathcal{C}_2=\gamma_{1}\cup\gamma_{3},
\end{split}
\end{equation*}
where $\gamma_1(z)=-iz$, $-R\le z\le R$, $\gamma_2(z)=Re^{i\varphi}$, $\pi/2\le \varphi \le 3\pi/2$, $\gamma_3(z)=R e^{i(\pi-\varphi)}$, $\pi/2\le \varphi \le 3\pi/2$,
such that $R>\max_{1\le \ell \le m, k\le j\le n}\{1/\beta_{\ell},1/\alpha_j\}$.
\end{lemma}

\begin{proof}
By definition of $g_{n,m,k}(t)$ the first stated asseration easily follows. By Lemma~\ref{GENSAMPlem2} $g_{n,m,k}(t)$ is a meromorphic function in $\C$, with simple poles at the values $\{-1/\beta_{\ell}\mid 1\le \ell\le m\}$ and $\{1/\alpha_j\mid k\le j\le n\}$. By taking the limit $R\to\infty$ in the two different contours of integration
we observe that the contributions of $\gamma_2$ and $\gamma_3$ are zero in the limit, and only the contribution of $\gamma_1$ 
remains. Hence, 
\begin{equation*}
\begin{split}
\frac{1}{2\pi i}\int_{-i\infty}^{i\infty}g_{n,m,k}(t)dt &=\frac{1}{2\pi i}\int_{\mathcal{C}_1}g_{n,m,k}(t)dt=\frac{1}{2\pi i}\int_{\mathcal{C}_2}g_{n,m,k}(t)dt.
\end{split}
\end{equation*}
\end{proof} 
Next we will derive an alternative representations for $G_{n,m,k}$ using the formal residue calculus.
\begin{lemma}
\label{GENSAMPlemOK3}
Assume that the weight sequence $A=(\alpha_n)_{n\in\N}$ satisfies $\alpha_i\neq \alpha_j$ for all $i,j\in\N$ with $i\neq j$.
Then the value $G_{n,m,k}$ admits for $n\ge k\ge 1$ and $m,\ge 1$ the following representation.
\begin{equation*}
G_{n,m,k}=\xi_k \alpha_k \sum_{j=k}^{n} \frac{\alpha_j^{m+n-k-1}}{\Big(\prod_{\substack{\ell=k\\\ell\neq j}}^{n}(\alpha_j-\alpha_{\ell})\Big)
\Big(\prod_{h=1}^{m}(\alpha_j+\beta_h)\Big)}.
\end{equation*}
Assume that the weight sequence $B=(\beta_m)_{m\in\N}$ satisfies $\beta_i\neq \beta_j$ for all $i,j\in\N$ with $i\neq j$.
Then the value $G_{n,m,k}$ admits for $n\ge k\ge 1$ and $m,\ge 1$ the following representation.
\begin{equation*}
G_{n,m,k}= \sum_{\ell=1}^{m} \frac{\beta_{\ell}^{n+m-1-k}}{\Big(\prod_{j=k}^{n}(\beta_{\ell}+\alpha_{j})\Big)
\Big(\prod_{\substack{h=1\\h \neq \ell}}^{m}(\beta_{\ell}-\beta_h)\Big)}.
\end{equation*}
\end{lemma}

\begin{proof}
By definition of $G_{n,m,k}$ and Lemma~\ref{GENSAMPlem2} we obtain the stated results by the residue theorem.
\end{proof}
We adapt the parameter $\xi_k$ in the representation of $G_{n,m,k}$ to the initial conditions $\P\{X_{n,0}=k\}=\delta_{n,k}$ for $n\ge 1$ and $k\ge 0$
and obtain after a few calculations the necessary condition $\xi_k=\alpha_k$ for $k\ge 1$.

\begin{lemma}
\label{GENSAMPlemOK4}
For $n,m\ge 1$ and $k\ge 0$ the values $\P\{X_{n,m}=k\}=\frac{1}{2\pi i}\int_{\mathcal{C}_3}\hat{g}_{n,m,k}(t)dt$, with 
\begin{equation*}
\hat{g}_{n,m,k}(t)=\begin{cases}
\frac{\alpha_k \Big(\prod_{j=1}^{k-1}(1- \alpha_j t )\Big)}{\Big(\prod_{j=1}^{n}(1-\alpha_j t)\Big)\Big(\prod_{\ell=1}^{m}(1+ \beta_{\ell} t\Big)},\quad \text{for}\quad
k\ge 1,\\[0.5cm]
\frac{1}{(-t)\Big(\prod_{j=1}^{n}(1- \alpha_j t )\Big)\Big(\prod_{\ell=1}^{m}(1+ \beta_{\ell} t)\Big)},\quad \text{for}\quad k=0,
\end{cases}
\end{equation*}
where $\mathcal{C}_3=\gamma_{1}'\cup\gamma_{2}'\cup\gamma_{3}'\cup\gamma_{4}$, $\gamma_1'(z)=-iz$, $-R\le z\le \varepsilon$, $\gamma_2'(z)=\varepsilon e^{i(3\pi/2-\varphi)}$, $0\le \varphi \le \pi$, $\gamma_3'(z)=iz$, $\varepsilon\le z\le R$ and $\gamma_4'(z)=R e^{i(\pi-\varphi)}$, $\pi/2\le \varphi \le 3\pi/2$, such that $R>\max_{1\le \ell \le m, k\le j\le n}\{1/\beta_{\ell},1/\alpha_j\}$ and $\varepsilon<\min_{1\le \ell \le m, k\le j\le n}\{1/\beta_{\ell},1/\alpha_j\}$,
are the proper solution of the system of recurrence relations~\eqref{GENSAMPrec1} fulfilling the stated initial conditions involding the boundary cases $n=0$ and $m=0$.
\end{lemma}

\begin{proof}
The initial conditions $\P\{X_{n,0}=k\}=\delta_{k,n}$ for $n\ge 1$ and $k\ge 1$ are satisfied since $\xi_k=\alpha_k$. 
Finally, we check the initial conditions $\P\{X_{0,m}=k\}=\delta_{k,0}$ for $m\in\N$ and $k\in\N_0$.
For $k=0$ we otain $\P\{X_{0,m}=0\}=1$ due to the additional simple pole at $t=0$. Concerning the remaining cases $k\ge 1$ we rely on the integral representation of the probabilities $\frac{1}{2\pi i}\int_{\mathcal{C}_3}\hat{g}_{n,m,k}(t)dt$. By Lemma~\ref{GENSAMPlem1} $\hat{g}_{n,m,k}(t)$ has no singularities with negative real part (more precisely, it has removeable singularities at $\{1/\alpha_j\mid 1\le j\le n\}$). Hence, by the Cauchy integral theorem $\P\{X_{0,m}=k\}=\frac{1}{2\pi i}\int_{\mathcal{C}_3}\hat{g}_{0,m,k}(t)dt=0$ for $k\ge 1$ and $n=0$.
\end{proof}

Finally, the probabilities $\P\{X_{n,m}=0\}$, with $n,m\ge 0$ are calculated in two different ways using the residue theorem.

\begin{proof}[Proof of Corollary 2]
In order to obtain the results for the P\'olya-Eggenberger generalized sampling urn model with ball transition matrix $M = \bigl(\begin{smallmatrix} 0 & -b \\ -c & 0\end{smallmatrix}\bigr)$, 
we adapt the weights sequences $(A,B)$ to the special case $A=(c\cdot n)_{n\in\N}$ and $B=(b\cdot m)_{m\in\N}$, with $a,b\in\N$.
Concerning the moments of $\hat{Y}_{cn,bm}=Y_{cn,bm}/c$ we proceed as follows.
We obtain the ordinary moments $\E(\hat{Y}_{cn,bm}^s)$ with $s\ge 1$, using the explicit expressions for the probabilities $\P\{Y_{cn,bm}=ck\}$ by summation. 
\begin{equation*}
\label{GENOKmoms}
\E\big(\hat{Y}_{cn,bm}^s\big)
=\sum_{k=1}^{n}k^s\P\{Y_{cn,bm}=ck\}= \sum_{k=1}^{n}\frac{k^{s+1}}{(n-k)!m!} \frac{c^{m}}{b^{m}} 
    \sum_{\ell=0}^{n} (-1)^{n-\ell} \frac{\binom{n-k}{\ell-k}}{\binom{m+\frac{c\ell}b}{m}}\ell^{m+n-1-k} .
\end{equation*} 
Interchanging the order of summation leads to 
\begin{equation*}
\begin{split}
\E\big(\hat{Y}_{cn,bm}^s\big)&=  \frac{c^{m}}{b^{m}} \sum_{\ell=0}^{n} (-1)^{n-\ell}
 \frac{\ell^{m+n}}{\binom{m+\frac{c\ell}b}{m} (n-\ell)!}\ell^{m+n} \sum_{k=1}^{\ell}\frac{k^{s+1}\ell^{-k-1}}{(\ell-k)!} \\
 &=\frac{1}{(n+m)!} \frac{c^{m}}{b^{m}} 
    \sum_{\ell=0}^{n} (-1)^{n-\ell}
    \frac{ \binom{n+m}{n-\ell}\binom{m+\ell}{\ell} }{\binom{m+\frac{c\ell}b}{m}}l\ell^{m+n}  \sum_{k=1}^{\ell}\binom{\ell-1}{k-1}k!\ell^{-k}k^s.
\end{split}
\end{equation*}
The second sum was studied in detail in~\cite{PuyPhD}. We will use the following important result. 
\begin{lemma}[Puyhaubert~\cite{PuyPhD}]
\label{THEpuy}
Let the functions $F(z,u)$ and $W(z,u)$ be defined as 
\begin{equation*}
F(z,u)= e^{u\big(e^{-z}+z-1\big)},\qquad G(z,u)=e^{u\big(e^{-z}+z-1\big)}\int_0^z u e^{-t}e^{-u\big(e^{-t}+t-1\big)}dt.
\end{equation*}
Furthermore, let $f_n(u)=n![z^n]F(z,u)$ and $g_n(u)=n![z^n]G(z,u)$, where the polynomial $f_n$ has degree $\lfloor \frac{n}{2}\rfloor$ and 
$g_n$ has degree $\lfloor \frac{n+1}{2}\rfloor$. The leading coefficient of $f_{2n}$ is given by $(2n-1)!!=(2n-1)(2n-3)\cdots 3$ and the leading coefficient
of $g_{2n+1}$ is given by $(2n)!!=(2n)(2n-2)\cdots 2$. Moreover, let $Q$ denote Ramanujan's Q-function 
defined as 
\begin{equation*}
Q(n)=1+\frac{n}n+\frac{n(n-1)}{n^2}+\cdots +\frac{n!}{n^n}=\sum_{i=0}^n\frac{\fallfak{n}{i}}{n^i}.
\end{equation*} 
The sum $\sum_{k=1}^{\ell}\binom{\ell-1}{k-1}k!\ell^{-k}k^s$ can be represented in terms of $f_n$, $g_n$ and $Q$ as follows.
\begin{equation*}
\sum_{k=1}^{\ell}\binom{\ell-1}{k-1}k!\ell^{-k}k^s=\frac{1}{\ell}(f_{s+1}(\ell)Q(l)+g_{s+1}(\ell)).
\end{equation*} 
\end{lemma}
The result above is proven in Puyhaubert~\cite{PuyPhD}, p.~117--118, we refer the reader to the proof therein. 
As pointed out in~\cite{PuyPhD} the function $F(z,u)$ is closely related to Stirling numbers of the second kind.
Note that in \cite{PuyPhD} there is a tiny misprint, namely the factor $n!$ in the definition of $f_n$ and $g_n$ is missing. For a comprehensive discussion of Ramanujan's Q-function we refer the reader to the work of Flajolet et al.\cite{Flajo1995}. Using Lemma~\ref{THEpuy} we immediately obtain the stated description of the moments of of $\hat{Y}_{cn,bm}$.
Furthermore, the second assertion follows easily by comparison of the degrees in $f_n(X)$ and $g_n(X)$, similar to~\cite{PuyPhD}.
The second expression for $\E(M_s(Y_{cn,bm}))$ is easily obtained by simple manipulations of the binomial coefficients.
\end{proof}

\section{Extension to higher dimensional models}
It is natural to ask whether the simple explicit expressions for the probabilities $\P\{X_{n,m}=k\}$, as given in Theorem~\ref{GENthe1}, can be extended to higher dimensional models, i.e.~involving more than two balls. In the following we will generalize urn models I and II by considering urns with $r\ge 2$ different types of balls.

\smallskip

At the beginning, the urn contains $n_{\ell}$ balls of type $\ell$ balls, with $n_{\ell}\ge 1$ for $1\le \ell\le r$.
Subsequently, we will use the multiindex notation $\mathbf{n}=(n_1,\dots,n_r)$.
Associated to the urn model are $r$ sequences of positive numbers $A^{[\ell]}=(\alpha^{[\ell]}_{n_{\ell}})_{n_{\ell}\in\N}$, for $1\le \ell\le r$.
At every step, we draw a ball from the urn according to the numbers of balls present in the urn with respect to the sequences $(A^{[\ell]})_{1\le\ell\le r}$, subject to the two models defined below. The choosen ball is discarded and the sampling procedure continues until type $r$ of balls is completely drawn. We are interested in the joint distribution of type 1 up to type $r-1$ balls when all type $r$ balls have been drawn.This corresponds again to the absorbing region $\mathcal{A} = \mathcal{A}_1\cup \mathcal{A}_2$, with 
$\mathcal{A}_1=\{(n_{1}, \dots, n_{r-1}, 0): n_{1}, \dots, n_{r-1} \ge 0\}$ and $\mathcal{A}_2=\{(0, \dots, 0, n_r): n_{r} \ge 1\}$
We study the random vector $\mathbf{X}_{\mathbf{n}}=(X_{\mathbf{n}}^{[1]},\dots,X_{\mathbf{n}}^{[r-1]})$,
where $X_{\mathbf{n}}^{[\ell]}=X_{n_{1}, \dots, n_{r}}^{[\ell]}$ counts the number of type $\ell$ balls, starting with $\mathbf{n}$ balls of types $1,\dots,r$, 
i.~e.~ $n_{\ell}$ balls of $\ell$ units, $1\le \ell \le r$. Subsequently, we will derive explicit expressions for the probabilities $\P\{\mathbf{X}_{\mathbf{n}}=\mathbf{k}\}$,
for $\mathbf{k}=(k_1,\dots,k_{r-1})$, with $k_{j}\ge 0$, for the following two urn models, extending our previous results corresponding to the case $r=2$.
\begin{itemize}
\item \emph{Urn model I}. Assume that the urn contains $n_{\ell}$ balls of type $\ell$ balls, $1\le \ell\le r$, with $n_r>0$ and at least one of the $n_{j}$
greater than zero, $1\le j\le r-1$. A ball of type $\ell$ is drawn with probability $\alpha_{n_{\ell}}^{[\ell]}/\sum_{j=1}^{r}\alpha_{n_{j}}^{[j]}$, $1\le \ell \le r$, with the additional assumption 
that $\alpha_0^{[\ell]}=0$ for $1\le \ell\le r$.

\item \emph{Urn model II}. Assume that the urn contains $n_{\ell}$ balls of type $\ell$ balls, $1\le \ell\le r$, with $n_r>0$ and at least one of the $n_{j}$
greater than zero, $1\le j\le r-1$. A ball of type $\ell$ is drawn with probability $(1-\delta_{n_{\ell},0})\prod_{\substack{j=1\\j\neq \ell}}^r (\alpha_{n_{j}}^{[j]})^{1-\delta_{n_j,0}}/\sum_{h=1}^{r}(1-\delta_{n_{h},0})\prod_{\substack{j=1\\j\neq h}}^r (\alpha_{n_{j}}^{[j]})^{1-\delta_{n_j,0}}$.
\end{itemize}

Note that urn model I is the natural extension of the urn model I with two types balls. 
By suitable choice of the sequences $A^{[\ell]}=(\alpha^{[\ell]}_{n_{\ell}})_{n_{\ell}\in\N}=(a_{\ell}\cdot n_{\ell})_{n_{\ell}\in\N}$, for $1\le \ell\le r$, 
we obtain the standard P\'olya-Eggenberger urn with $r$ different kind of balls, and ball transition matrix $M$ given by 
\begin{small}
\begin{equation}
\label{GENhighpol}
      M=\left(
    \begin{smallmatrix}
         -a_1 & 0 &  \cdots & 0 &0 \\
         0 & -a_2 & 0 & \ddots   & 0 \\
         0 & 0 &  \ddots & \ddots  & 0 \\
         \vdots & \ddots & \ddots & \ddots & \vdots \\
         0 &  \ddots & \ddots &  0& -a_{r}  \\
    \end{smallmatrix}
      \right).
   \end{equation}
   \end{small}
In contrast the extension of urn model II is non-standard in the sense
that for any specification of the sequences  $A^{[\ell]}=(\alpha^{[\ell]}_{n_{\ell}})_{n_{\ell}\in\N}$, $1\le \ell\le r$, we do not obtain 
ordinary higher dimensional P\'olya-Eggenberger urn models anymore. The reason for this will become clear in Section~\ref{GENsecdual}, where we uncover the duality relation between these two urn models.

\subsection{Explicit results for the probabilities}
\begin{theorem}
\label{GENthehigh1}
The distribution of the random vector $\mathbf{X}_{\mathbf{n}}$, counting the number of remaining type $1,\dots,r-1$ balls
when all type $r$ balls have been drawn in urn model I, is for $n_j\ge 1$ and $n_j\ge k_j\ge 0$, with $1\le j\le r$, given by the explicit formula 
\begin{equation*}
\begin{split}
\P\{\mathbf{X}_{\mathbf{n}}=\mathbf{k}\}&= \sum_{\ell_1=k_1}^{n_1}\dots\sum_{\ell_{r-1}=k_{r-1}}^{n_{r-1}}
\frac{\Big(\prod_{f=1}^{n_r}\alpha_f^{[r]}\Big)  \prod_{j=1}^{r-1}\bigg(\prod_{h_j=k_j+1}^{n_j}\alpha_{h_j}^{[j]}\bigg) }
{\Big(\prod_{f=1}^{n_r}(\alpha_f^{[r]}+\sum_{j=1}^{r-1}\alpha_{\ell_j}^{[j]}) \Big)\prod_{j=1}^{r-1}\Big(\prod_{\substack{h_j=k_j\\h_j\neq \ell_j}}^{n_j}(\alpha_{h_j}^{[j]}-\alpha_{\ell_j}^{[j]})\Big) }
\end{split}
\end{equation*}
assuming that $\alpha_h^{[j]} \neq \alpha_{\ell}^{[j]}$ , $1\le h<\ell<\infty$, $1\le j\le r$.
\end{theorem}

\begin{coroll}
\label{GENcollhigh1}
The distribution of the random vector $\mathbf{Y}_{\mathbf{an}}=(Y_{\mathbf{an}}^{[1]},\dots, Y_{\mathbf{an}}^{[r-1]})$, counting the number of type $\ell$ balls, $1\le \ell\le r-1$, in the sampling without replacement 
P\'olya-Eggenberger urn model with $r$ different types of balls~\eqref{GENhighpol}, when all type $r$ balls have been drawn, starting with $a_{j}n_j$ balls of type 
$j$, $1\le j\le r$, is given by
\begin{equation*}
\begin{split}
\P\{\mathbf{Y}_{\mathbf{an}}=\mathbf{ak}\}&= \sum_{\ell_1=k_1}^{n_1}\dots\sum_{\ell_{r-1}=k_{r-1}}^{n_{r-1}}
\frac{\prod_{j=1}^{r-1}\binom{n_j}{\ell_j}\binom{\ell_j}{k_j}(-1)^{\ell_j-k_j}}{\binom{n_r+\sum_{f=1}^{r-1}\frac{a_f\ell_j}{a_r}}{n_r}}.
\end{split}
\end{equation*}
The mixed factorial moments $\E\bigg(\prod_{j=1}^{r-1}\fallfak{\Big(\frac{Y_{\mathbf{an}}}{a_j}^{[j]}\Big)}{s_j}\bigg)$ of $\mathbf{Y}_{\mathbf{an}}$ are given by the exact formula
\begin{equation*}
\E\bigg(\prod_{j=1}^{r-1}\fallfak{\Big(\frac{Y_{\mathbf{an}}}{a_j}^{[j]}\Big)}{s_j}\bigg)= \frac{\prod_{j=1}^{r-1}\fallfak{n_j}{s_j}}{\binom{n_r+\sum_{f=1}^{r-1}\frac{a_f s_j}{a_r}}{n_r}}.
\end{equation*}

\end{coroll}

\begin{theorem}
\label{GENthehigh2}
The distribution of the random vector $\mathbf{X}_{\mathbf{n}}$, counting the number of remaining type $1,\dots,r-1$ balls
when all type $r$ balls have been drawn in urn model II, is for $n_j\ge 1$ and $n_j\ge k_j\ge 1$, with $1\le j\le r$, given by the explicit formula 
\begin{equation*}
\begin{split}
\P\{\mathbf{X}_{\mathbf{n}}=\mathbf{k}\}&= \sum_{\ell_1=k_1}^{n_1}\dots\sum_{\ell_{r-1}=k_{r-1}}^{n_{r-1}}
\frac{\bigg(\prod_{j=1}^{r-1}\alpha_{k_j}^{[j]}\bigg)\prod_{j=1}^{r-1}(\alpha_{\ell_j}^{[j]})^{n_j-k_j+n_r-1}}{  
\prod_{f=1}^{n_r}\Big( \prod_{j=1}^{r-1}\alpha_{\ell_j}^{[j]} + \alpha_f^{[r]}\sum_{g=1}^{r-1}\frac{\prod_{j=1}^{r-1}\alpha_{k_j}^{[j]} }{\alpha_{\ell_g}^{[g]}}\Big)
\prod_{j=1}^{r-1}\prod_{\substack{h_j=k_j\\h_j\neq \ell_j}}^{n_j}(\alpha_{\ell_j}^{[j]}-\alpha_{h_j}^{[j]}) },
\end{split}
\end{equation*}
assuming that $\alpha_h^{[j]} \neq \alpha_{\ell}^{[j]}$ , $1\le h<\ell<\infty$, $1\le j\le r$. 
\end{theorem}
Similar but slightly more involved expressions exist for urn model II in the cases $k_j=0$, for $1\le j\le r-1$.

\subsection{Deriving the explicit results}
Subsequently, we will briefly discuss the proof of Theorem~\ref{GENthehigh1} and comment on the proofs of Corollary~\ref{GENcollhigh1} of Theorem~\ref{GENthehigh2}, which are extensions of the proofs of the two dimensional cases, $r=2$. By definition of urn model I the probabilities $\P\{\mathbf{X}_{\mathbf{n}}=\mathbf{k}\}$ satisfy the following system of recurrence relations for $n_r\ge 1$, $n_h\ge 0$, $1\le h\le r-1$, with at least one of the $n_h>0$.
\begin{equation}
\label{GENhighrec}
\P\{\mathbf{X}_{\mathbf{n}}=\mathbf{k}\}= \sum_{\ell=1}^{r}\frac{\alpha_{n_{\ell}}^{[\ell]}}{\sum_{j_{\ell}=1}^{r}\alpha_{n_j}^{[j]}}\P\{\mathbf{X}_{\mathbf{n}-\mathbf{e}_{\ell}}=\mathbf{k}\}.
\end{equation}
where $\mathbf{e}_{\ell}$ denotes the $\ell$-th unit vector, with respect to the additional assumption that $\alpha_0^{[\ell]}=0$ for $1\le \ell\le r$.
The initial conditions are given by 
\begin{equation}
\label{GENhighini}
\begin{split}
\P\{X_{n_1,\dots,n_{r-1},0}=\mathbf{k}\}&=\prod_{j=1}^{r-1}\delta_{k_j,n_j},\quad n_j\ge 0,\,k_j\ge 0,\quad\text{with at least one }\,\,n_j>0,\\
\P\{X_{0,\dots,0,n_r}=\mathbf{k}\}&=\prod_{j=1}^{r-1}\delta_{k_j,0},\quad n_r\ge 1,\,k_j\ge 0.
\end{split}
\end{equation}

In order to prove the result above we first derive a class of formal solution of the system of recurrence relations.
Then, out of the class of formal solutions we determine the proper solution by adapting to the initial conditions.
The crucial part of the proof is the following result.

\begin{lemma}
\label{GENhighlem1}
The quantities $f_{\mathbf{n},\mathbf{k}}(\mathbf{t})\in\C$, defined as 
\begin{equation*}
f_{\mathbf{n},\mathbf{k}}(\mathbf{t})=
\frac{\xi_{\mathbf{k}}(\mathbf{t}) \Big(\prod_{\ell=1}^{r-1}\prod_{j_{\ell}=1}^{k_{\ell}-1}(1-\frac{t_{\ell}}{\alpha_j^{[\ell]}})\Big)}{\Big(\prod_{\ell=1}^{r-1}\prod_{j_{\ell}=1}^{n_{\ell}}(1-\frac{t_{\ell}}{\alpha_j^{[\ell]}})\Big)
\Big(\prod_{h=1}^{n_r}(1+\frac{\sum_{f=1}^{r-1}t_f}{\alpha_{h}^{[r]}})\Big)},
\end{equation*}
with parameter $\xi_{\mathbf{k}}(\mathbf{t})$, satisfy for each fixed $t_{\ell}\in i\cdot\R$ and $n_{\ell}\ge 1$ and $k_{\ell}\ge 1$, for $1\le \ell\le r-1$, and $n_r\ge 1$, a recurrence relation of the type~\eqref{GENhighrec}. 
\end{lemma}
In order to validate Lemma~\ref{GENhighlem1} we observe that
\begin{equation*}
\begin{split}
\sum_{\ell=1}^{r}\frac{\alpha_{n_{\ell}}^{[\ell]}}{\sum_{j=1}^{r}\alpha_{n_j}^{[j]}}f_{\mathbf{n}-\mathbf{e}_{\ell},\mathbf{k}}(\mathbf{t})
&=\sum_{\ell=1}^{r-1}\frac{\alpha_{n_{\ell}}^{[\ell]}}{\sum_{j=1}^{r}\alpha_{n_j}^{[j]}}   
f_{\mathbf{n},\mathbf{k}}(\mathbf{t})\Big(1-\frac{t_{\ell}}{\alpha_{n_{\ell}}^{[\ell]}}\Big)
+ \frac{\alpha_{n_{r}}^{[r]}}{\sum_{j=1}^{r}\alpha_{n_j}^{[j]}}f_{\mathbf{n},\mathbf{k}}(\mathbf{t})\Big(1+\frac{\sum_{f=1}^{r-1}t_{f}}{\alpha_{n_{r}}^{[r]}}\Big)\\
&= \frac{1}{\sum_{j=1}^{r}\alpha_{n_j}^{[j]}}f_{\mathbf{n},\mathbf{k}}(\mathbf{t})\bigg( \sum_{\ell=1}^{r-1}\alpha_{n_{\ell}}^{[\ell]}\Big(1-\frac{t_{\ell}}{\alpha_{n_{\ell}}^{[\ell]}}\Big)  +\alpha_{n_{r}}^{[r]}\Big(1+\frac{\sum_{f=1}^{r-1}t_{f}}{\alpha_{n_{r}}^{[r]}}\Big)\bigg)=f_{\mathbf{n},\mathbf{k}}(\mathbf{t}).
\end{split}
\end{equation*}
As in the case $r=2$ we derive another family of solutions using the Cauchy integral theorem.

\begin{lemma}
\label{GENhighlem2}
The values $F_{\mathbf{n},\mathbf{k}}=\frac{1}{(2\pi i)^{r-1}}\int_{\mathcal{C}_1}\dots\int_{\mathcal{C}_{r-1}}f_{\mathbf{n},\mathbf{k}}(\mathbf{t})dt_r\dots dt_1$ are for $n,m,k\in\N$ formal solutions of the system of recurrence relations~\eqref{GENhighrec}, with parameter $\xi_{\mathbf{k}}$ assumed to be a constant independent of $\mathbf{t}$, where the 
interior of the contours of integration $\mathcal{C_{\ell}}$ are containing the the values $\alpha_{h_{\ell}}^{[\ell]}$, $1\le \ell \le r-1$ and $k_{\ell}\le h_{\ell}\le n_{\ell}$, 
plus the origin, and the orientation of these closed curves is mathematically negative (clockwise).
\end{lemma}
It still remains to adapt to the initial conditions~\eqref{GENhighini}. It turns out that the choice 
\begin{equation*}
\xi_{\mathbf{k}}(\mathbf{t})=\frac{1}{\prod_{\ell=1}^{r-1} (\alpha_{k_{\ell}})^{1-\delta_{k_{\ell},0}}(-t_{\ell}^{\delta_{k_{\ell},0}})},
\end{equation*}
provides the right solution. Note that $\xi_{\mathbf{k}}(\mathbf{t})$ depends to $\mathbf{t}$, but only in the case of $k_{\ell}=0$, $1\le \ell \le r-1$.
We skip the long and involved calculations, which are in principle similar to calculations for $r=2$.

\smallskip

Concerning the result for the P\'olya-Eggenberger urn model with ball transition matrix $M$ given by~\eqref{GENhighpol} we
adapt the sequences $A^{[\ell]}=(\alpha^{[\ell]}_{n_{\ell}})_{n_{\ell}\in\N}=(a_{\ell}\cdot n_{\ell})_{n_{\ell}\in\N}$, for $1\le \ell\le r$,
and simplify the result of Theorem~\ref{GENthehigh1} to obtain the first statement of Corollary~\ref{GENcollhigh1}. 
The mixed factorial moments $\E\bigg(\prod_{j=1}^{r-1}\fallfak{\Big(\frac{Y_{\mathbf{an}}}{a_j}^{[j]}\Big)}{s_j}\bigg)$ of $\mathbf{Y}_{\mathbf{an}}$
are derived as follows.
\begin{equation*}
\begin{split}
\E\bigg(\prod_{j=1}^{r-1}\fallfak{\Big(\frac{Y_{\mathbf{an}}}{a_j}^{[j]}\Big)}{s_j}\bigg)
&=\sum_{k_1=0}^{n_1}\dots\sum_{k_{r-1}=0}^{n_{r-1}}\P\{\mathbf{Y}_{\mathbf{an}}=\mathbf{ak}\}\prod_{j=1}^{r-1}\fallfak{k_j}{s_j}\\
&=\sum_{\ell_1=s_1}^{n_1}\dots\sum_{\ell_{r-1}=s_{r-1}}^{n_{r-1}}\frac{\Big(\prod_{j=1}^{j-1}\binom{n_j}{s_j}s_j!\Big)}{\binom{n_r+\sum_{f=1}^{r-1}\frac{a_f\ell_j}{a_r}}{n_r}}
\prod_{j=1}^{r-1}\sum_{k_j=s_j}^{\ell_j}(-1)^{\ell_j-s_j}\binom{\ell_j}{k_j}\binom{k_j}{s_j}\\
&=\sum_{\ell_1=s_1}^{n_1}\dots\sum_{\ell_{r-1}=s_{r-1}}^{n_{r-1}}\frac{\Big(\prod_{j=1}^{j-1}\binom{n_j}{s_j}s_j!\Big)}{\binom{n_r+\sum_{f=1}^{r-1}\frac{a_f\ell_j}{a_r}}{n_r}}
\prod_{j=1}^{r-1}\delta_{\ell_j,s_j}=\frac{\prod_{j=1}^{r-1}\fallfak{n_j}{s_j}}{\binom{n_r+\sum_{f=1}^{r-1}\frac{a_f s_j}{a_r}}{n_r}},
\end{split}
\end{equation*}
where $\delta_{\ell,s}$ denotes the Kronecker-delta function.
The results for urn model II are derived analogously to the proof of Theorem~\ref{GENthehigh1} by first deriving a class of formal solution of the system of recurrence relations and 
then adaptin to the initial conditions.

\section{Duality of the two urn models\label{GENsecdual}}
One can notice similarities between the explicit results for the probabilities $\P\{X_{n,m}=k\}$ for urn models I and II in the case $r=2$ and also the similarities of the proofs.
However, the higher dimensional models seem to be very different. We will show that urn models I and II are closely connected to each other, as they satisfy a certain duality relation,
and are essentially the same urn model, also for the higher dimensional cases. 

\begin{theorem}
Let $\P\{X_{n,m,[A,B,I]}=k\}$ denote the probability that $k$ white balls remain when all black balls have been drawn in urn model I with weight
sequences $A=(\alpha_n)_{n\in\N}$, $B=(\beta_m)_{m\in\N}$ and $\P\{X_{n,m,[\tilde{A},\tilde{B},II]}=k\}$ the corresponding probability in urn model II with weight
sequences $\tilde{A}=(\tilde{\alpha}_n)_{n\in\N}$, $\tilde{B}=(\tilde{\beta}_m)_{m\in\N}$. The probabilities $\P\{X_{n,m,[A,B,I]}=k\}$ and $\P\{X_{n,m,[\tilde{A},\tilde{B},II]}=k\}$ are dual to each other, i.e.~ they are related in the following way.
\begin{equation*}
\P\{X_{n,m,[A,B,I]}=k\}=\P\{X_{n,m,[\tilde{A},\tilde{B},II]}=k\},\quad \text{for} \quad \alpha_n=\frac{1}{\tilde{\alpha}_n}
,\quad \beta_m =\frac{1}{\tilde{\beta}_m},\quad\text{for all}\quad n,m\in\N.
\end{equation*}
\end{theorem}

\begin{remark}
According to the result above, there is only one urn model lurking behind urn models I and II, which can be formulated either as an urn model of OK Corral type or of sampling without replacement type with general weight sequences.
\end{remark}

\begin{proof}
We observe that the recurrence relations, as stated in~\eqref{GENSAMPrec1},~\eqref{GENSAMPrec2} can be transformed in the following way.
\begin{equation*}
\begin{split}
\P\{X_{n,m,[A,B,I]}=k\}&=\frac{\alpha_n}{\alpha_n+\beta_m}\P\{X_{n-1,m,[A,B,I]}=k\} + \frac{\beta_m}{\alpha_n+\beta_m}\P\{X_{n,m-1,[A,B,I]}=k\}\\
&= \frac{\frac1{\beta_m}}{\frac1{\beta_m}+\frac1{\alpha_n}}\P\{X_{n-1,m,[A,B,I]}=k\} + \frac{\frac1{\alpha_n}}{\frac1{\beta_m}+\frac1{\alpha_n}}\P\{X_{n,m-1,[A,B,I]}=k\}\\
&= \frac{\tilde{\beta}_m}{\tilde{\beta}_m+\tilde{\alpha}_n}\P\{X_{n-1,m,[A,B,I]}=k\} + \frac{\tilde{\alpha}_n}{\tilde{\beta}_m+\tilde{\alpha}_n}\P\{X_{n,m-1,[A,B,I]}=k\}\\
&= \frac{\tilde{\beta}_m}{\tilde{\beta}_m+\tilde{\alpha}_n}\P\{X_{n-1,m,[\tilde{A},\tilde{B},II]}=k\} + \frac{\tilde{\alpha}_n}{\tilde{\beta}_m+\tilde{\alpha}_n}\P\{X_{n,m-1,[\tilde{A},\tilde{B},II]}=k\}\\
&=\P\{X_{n,m,[\tilde{A},\tilde{B},II]}=k\} .
\end{split}
\end{equation*}
Since the initial conditions coincide we can transform one urn model into the other, which proves the stated result.
\end{proof}

In particular, the duality relation explains why urn model II looks so different in higher dimensions .

\begin{theorem}
Let $\P\{\mathbf{X}_{\mathbf{n},[\mathbf{A},I]}=\mathbf{k}\}$ denote the probability that $k_j$ type $j$ balls remain when type $r$ balls have been drawn in urn model I with weight
sequences $\mathbf{A}=(A_1,\dots,A_r)$ with $A_j=(\alpha_{n_j}^{[j]})_{n_j\in\N}$, and $\P\{\mathbf{X}_{\mathbf{n},[\tilde{\mathbf{A}},II]}=\mathbf{k}\}$ the corresponding probability in urn model II with weight sequences $\tilde{\mathbf{A}}=(\tilde{A}_1,\dots,\tilde{A}_r)$ with $\tilde{A}_j=(\tilde{\alpha}_{n_j}^{[j]})_{n_j\in\N}$. The probabilities $\P\{\mathbf{X}_{\mathbf{n},[\mathbf{A},I]}=\mathbf{k}\}$ and $\P\{\mathbf{X}_{\mathbf{n},[\tilde{\mathbf{A}},II]}=\mathbf{k}\}$ are dual to each other, i.e.~ they are related in the following way.
\begin{equation*}
\P\{\mathbf{X}_{\mathbf{n},[\mathbf{A},I]}=\mathbf{k}\}=\P\{\mathbf{X}_{\mathbf{n},[\tilde{\mathbf{A}},II]}=\mathbf{k}\},\quad \text{for} \quad \alpha_{n_j}^{[j]}=\frac{1}{\tilde{\alpha}_{n_j}^{[j]}}
,\quad\text{for all}\quad n_j\ge 1,\quad\text{with}\quad 1\le j\le r.
\end{equation*}
\end{theorem}

\begin{proof}
We observe that the recurrence relations, as stated in~\eqref{GENSAMPrec1},~\eqref{GENSAMPrec2} can be transformed in the following way.
\begin{equation*}
\begin{split}
\P\{\mathbf{X}_{\mathbf{n},[\mathbf{A},I]}=\mathbf{k}\}&=
\sum_{\ell=1}^{r}\frac{\alpha_{n_{\ell}}^{[\ell]}}{\sum_{j_{\ell}=1}^{r}\alpha_{n_j}^{[j]}}\P\{\mathbf{X}_{\mathbf{n}-\mathbf{e}_{\ell},[\mathbf{A},I]}=\mathbf{k}\}=\sum_{\ell=1}^{r}\frac{\frac{\alpha_{n_{\ell}}^{[\ell]}}{\prod_{h=1}^{r}\alpha_{n_h}^{[h]}}}{\sum_{j=1}^{r}\frac{\alpha_{n_j}^{[j]}}{\prod_{h=1}^{r}\alpha_{n_h}^{[h]}}}\P\{\mathbf{X}_{\mathbf{n}-\mathbf{e}_{\ell},[\mathbf{A},I]}=\mathbf{k}\}\\
&=\sum_{\ell=1}^{r}\frac{\prod_{\substack{j=1\\j\neq \ell}}^{r}\tilde{\alpha}_{n_j}^{[j]}}{\sum_{j=1}^{r}\prod_{\substack{h=1\\h\neq j}}^{r}\tilde{\alpha}_{n_h}^{[h]}}\P\{\mathbf{X}_{\mathbf{n}-\mathbf{e}_{\ell},[\mathbf{A},I]}=\mathbf{k}\}
=\sum_{\ell=1}^{r}\frac{\prod_{\substack{j=1\\j\neq \ell}}^{r}\tilde{\alpha}_{n_j}^{[j]}}{\sum_{j=1}^{r}\prod_{\substack{h=1\\h\neq j}}^{r}\tilde{\alpha}_{n_h}^{[h]}}\P\{\mathbf{X}_{\mathbf{n}-\mathbf{e}_{\ell},[\tilde{\mathbf{A}},II]}=\mathbf{k}\}\\
&=\P\{\mathbf{X}_{\mathbf{n},[\tilde{\mathbf{A}},II]}=\mathbf{k}\},
\end{split}
\end{equation*}
for $n_j>0$, $1\le j\le r$. One still has to take into accounts the various cases of $n_j =0$, $1\le j\le r-1$ which leads to the Kronecker delta terms in description of urn model two.
Since the initial conditions coincide we can transform one urn model into the other, which proves the stated result.
\end{proof}

\section{Discussion of a special case\label{GENSAMPthetasec}}
In this section we study in more detail a special case of urn model I with respect to the choices
$A=(\alpha_n)_{n\in\N}=(n)_{n\in\N}$ and $B=(\beta_m)_{m\in\N}=(m^2)_{m\in\N}$ for the weight sequences $(A,B)$. 
We are interested in the various limiting distributions arising for $\max\{n,m\}\to\infty$. Interestingly, the Jacobi Theta function $\Theta(q)=\sum_{n\in\Z}(-1)^nq^{n^2}$ naturally arises.

\begin{theorem}
\label{GENSAMPtheTheta1}
The random variable $X_{n,m}$, counting the number of remaining white balls when all black balls
have been drawn in urn model I with weight sequences $A=(n)_{n\in\N}$, $B=(m^2)_{m\in\N}$, satisfies the following limiting laws.
\begin{enumerate}
	\item Case arbitrary but fixed $m\in\N$ and $n\to\infty$: $X_{n,m}/n$ converges in distribution 
	to a continuous random variable $Y_m$ with support $[0,1]$:
	\begin{equation*}
	\frac{X_{n,m}}{n}\claw Y_m,\qquad \P\{Y_m\le x\}= \int_{0}^{x} f_m(q)dq, \quad 0\le x \le 1,
	\end{equation*}
	with density function $f_m(q)$ being given by
	\begin{equation*}
	f_m(q)=2\sum_{\ell=1}^{m}(-1)^{\ell-1}\frac{\binom{m}{\ell}}{\binom{m+\ell}{m}}\ell^2 q^{\ell^2-1}.
	\end{equation*}
	We have also have convergence of all positive integer moments of $X_{n,m}$,
	\begin{equation*}
\E\Big(\frac{X_{n,m}}{n}\Big)^s\to\E(Y_m^s),\qquad	\E(Y_m^s)= \frac{(m!)^2}{\Gamma(m+1-i\sqrt{s})\Gamma(m+1+i\sqrt{s})}\cdot \frac{\pi\sqrt{s}}{\sinh{\pi\sqrt{s}}},
	\end{equation*}
	where $i$ denotes the imaginary unit $i^2=-1$.
	\item Case arbitrary but fixed $n\in\N$ and $m\to\infty$: $X_{n,m}$ converges in distribution
	to a discrete random variable $Z_{n}$ 
  \begin{equation*}
	X_{n,m} \claw Z_n,\qquad \P\{Z_n=k\}=\sum_{\ell=k}^{n}(-1)^{\ell-k}\binom{n}{\ell}\binom{\ell}{k}\frac{\pi\sqrt{\ell}}{\sinh{\pi\sqrt{\ell}}}
	=\sum_{\ell=1}^{\infty}(-1)^{\ell-1}\frac{\binom{k-1+\ell^2}{k}}{\binom{n+\ell^2}{n}},
	\quad 0\le k\le n.
	\end{equation*}
	We have also have convergence of all positive integer moments of $X_{n,m}$,
		\begin{equation*}
\E(X_{n,m}^s)\to\E(Z_n^s),\qquad	\E(Z_n^s)= \sum_{\ell=1}^{s}n^{\ell}\sum_{j=\ell}^{s}\Stir{s}{j}\stir{j}{\ell}(-1)^{j-\ell}\frac{\pi\sqrt{j}}{\sinh{\pi\sqrt{j}}},
	\end{equation*}

\item Case $\min\{n,m\}\to\infty$: $X_{n,m}/n$ converges in distribution to a random variable $W$ with support $[0,1]$, whose distribution 
function can be expressed in terms of the Jacobi Theta function $\Theta(q)=\sum_{n\in\Z}(-1)^nq^{n^2}$,
\begin{equation*}
\frac{X_{n,m}}{n}\claw W,\qquad \P\{W\le q\} = 1-\Theta(q),\quad 0\le q \le 1.
\end{equation*}
	We have also have convergence of all positive integer moments of $X_{n,m}$,
	\begin{equation*}
\E\Big(\frac{X_{n,m}}{n}\Big)^s\to\E(W^s),\qquad	\E(W^s)= \frac{\pi\sqrt{s}}{\sinh{\pi\sqrt{s}}}.
	\end{equation*}
\end{enumerate}
\end{theorem}
We can slightly extend the previous theorem by giving a more detailed description of the limit laws by studying 
the limiting distributions of the random variables $Y_m$ and $Z_n$.
\begin{coroll}
\label{GENSAMPtheTheta2}
The limit laws of the random variable $X_{n,m}$, counting the number of remaining white balls when all black balls
have been drawn in urn model I with weight sequences $A=(n)_{n\in\N}$, $B=(m^2)_{m\in\N}$, 
can be described by the following diagram concerning the convergence in distribution of the various random variables.
\begin{table}[!htb]
\centering
\begin{tabular}{lcr}
$X_{n,m}$ 				    		& $\xrightarrow[\text{normalization }\frac1n]{n\to\infty}$ & $Y_m$\\[0.35cm]
$\Big\downarrow$ {\tiny $m\to\infty $} & {\tiny $\frac1n$} {\LARGE $\searrow$} {\tiny $^{\min\{n,m\}\to\infty}$}& {\tiny $m\to\infty$} $\Big\downarrow$ \\[0.35cm]
$Z_{n}$ 								& $\xrightarrow[\text{normalization }\frac1n]{n\to\infty} $& $W$
\end{tabular}
\end{table}
\end{coroll}

\begin{coroll}
\label{GENSAMPtheTheta3}
The random variables $Y_m$, with $m\in\N$, and $W$ of Theorem~\ref{GENSAMPtheTheta1} satify the following decompositions. 
\begin{equation*}
Y_m\law \exp\Big(-\sum_{\ell=1}^{m}\frac{\epsilon_{\ell}}{\ell^2}\Big),\qquad
W\law \exp\Big(-\sum_{\ell=1}^{\infty}\frac{\epsilon_{\ell}}{\ell^2}\Big),
\end{equation*}
where the $\epsilon_{\ell}$ are independent and identically distributed random variables with 
exponential distribution $\Exp(1)$.
\end{coroll}

\begin{proof}[Proofs of Theorem~\ref{GENSAMPtheTheta1}  and Corollary~\ref{GENSAMPtheTheta2}]
In the following we will outline the main steps of the proofs of Theorem~\ref{GENSAMPtheTheta1} and Corollary~\ref{GENSAMPtheTheta2}.
Our starting points are the explicit expressions for $P\{X_{n,m}=k\}$ as given in Theorem~\ref{GENthe1},
\begin{equation*}
\P\{X_{n,m}=k\}=(m!)^2\frac{n!}{k!} \sum_{\ell=1}^{m}\frac{1}{\prod_{j=k}^{n}(j+\ell^2)\prod_{\substack{i=1\\ i\neq \ell}}^{m}(i^2-\ell^2)}
=(m!)^2\frac{n!}{k!} \sum_{\ell=k}^{n}\frac{1}{\prod_{\substack{j=k\\j\neq \ell}}^{n}(j-\ell)\prod_{i=1}^{m}(i^2+\ell)}.
\end{equation*}
By routine manipulations we derive the alternative represenation
\begin{equation*}
\P\{X_{n,m}=k\}= 2\sum_{\ell=1}^{m}(-1)^{\ell-1} \frac{\binom{k+\ell^2-1}{k}\binom{m}{\ell}}{\binom{n+\ell^2}{n}\binom{m+\ell}{m}}
= \sum_{\ell=k}^{n}(-1)^{\ell-k}\binom{n}{\ell}\binom{\ell}{k}\prod_{j=1}^{m}\frac{j^2}{j^2+s}.
\end{equation*}
Concerning case (1) such that $m\in\N$ is fixed and $n\to\infty$ we proceed by deriving a local limit law,
i.e. we study $n\cdot\P\{X_{n,m}=k\}$ for $k=\lfloor n\cdot q\rfloor$ with $q\in[0,1]$.
After manipulation of the stated products we apply Stirling's formula for the Gamma function
\begin{equation}
\label{GENSAMPStirling}
    \Gamma(x) = \Bigl(\frac{x}{e}\Bigr)^{x}\frac{\sqrt{2\pi }}{\sqrt{x}}%
    \Bigl(1+\frac{1}{12x}+\frac{1}{288x^{2}}+\mathcal{O}(\frac{1}{x^{3}})\Bigr),
\end{equation}
to the quotient of Gamma functions
\begin{equation*}
\frac{\binom{k+\ell^2-1}{k}}{\binom{n+\ell^2}{n}}=\ell^2\frac{\Gamma(n+1)\Gamma(k+\ell^2)}{\Gamma(n+1+\ell^2)\Gamma(k+1)}= \frac{\ell^2q^{\ell^2-1}}{n}\big(1+\mathcal{O}(\frac1n)\big).
\end{equation*}
Thus we obtain a local limit theorem, which in turn implies the convergence in distribution after 
a few routine manipulations. In order to prove the moment convergence 
we study first the factorial moments $\E(\fallfak{X_{n,m}}{s})=\E(X_{n,m}(X_{n,m}-1)\dots(X_{n,m}-s+1))$.
We obtain a closed form expression in the following way.
\begin{equation*}
\E(\fallfak{X_{n,m}}{s})=\sum_{k=0}^{n}\fallfak{k}{s}\P\{X_{n,m}=k\}=
2\sum_{\ell=1}^{m}(-1)^{\ell-1} \frac{\binom{m}{\ell}}{\binom{m+\ell}{m}}\sum_{k=0}^{n}\fallfak{k}{s}\frac{\binom{k+\ell^2-1}{k}}{\binom{n+\ell^2}{n}},
\end{equation*}
with $s\ge 1$. We use the summation identities
\begin{equation*}
\sum_{k=1}^{n}\fallfak{k}{s}\binom{k+X}{k}=\frac{\fallfak{(n+1)}{s+1}\binom{n+1+X}{n+1}}{X+s+1},\qquad \sum_{\ell=1}^{m}\frac{1}{\ell^2+s}\prod_{\substack{i=1\\ i\neq \ell}}^{m}\frac{1}{i^2-\ell^2}
=\prod_{\ell=1}^{m}\frac{1}{s+\ell^2},
\end{equation*}
for $s\in\N$, where the second identity easily follows from the partial fraction decomposition~\eqref{GENSAMPparfrac}.
We obtain the simple result
\begin{equation*}
\E(\fallfak{X_{n,m}}{s})= 2\fallfak{n}{s}\sum_{\ell=1}^{m}(-1)^{\ell-1} \frac{\binom{m}{\ell}}{\binom{m+\ell}{m}}\frac{\ell^{2}}{\ell^2+s}
=\fallfak{n}{s} \prod_{\ell=1}^{m}\frac{\ell^2}{\ell^2+s}.
\end{equation*}
Consequently, we can write the factorial moments as polynomials in $n$ of degree $s$, 
\begin{equation*}
\E(\fallfak{X_{n,m}}{s})= \prod_{\ell=1}^{m}\frac{\ell^2}{\ell^2+s}\sum_{\ell=1}^{s}\stir{s}{\ell}(-1)^{s-\ell}n^{\ell},
\end{equation*}
where the $\stir{n}{k}$ denote the unsigned Stirling numbers of the first kind, also known as Stirling cycle numbers.
In order to obtain the ordinary moments we proceed as follows.
\begin{equation}
\label{GENSAMPfinalmom}
\E(X_{n,m}^s)=\sum_{j=1}^{s}\Stir{s}{j}\E(\fallfak{X_{n,m}}{j})
=\sum_{j=1}^{s}\Stir{s}{j}\prod_{h=1}^{m}\frac{h^2}{h^2+j}\sum_{\ell=1}^{j}\stir{j}{\ell}(-1)^{j-\ell}n^{\ell}
=\sum_{\ell=1}^{s}n^{\ell}\sum_{j=\ell}^{s}\Stir{s}{j}\stir{j}{\ell}(-1)^{j-\ell}\prod_{h=1}^{m}\frac{h^2}{h^2+j}.
\end{equation}
This implies that for $n\to\infty$ the ordinary moments satisfy the expansion 
\begin{equation*}
\E(X_{n,m}^s)=n^s\prod_{\ell=1}^{m}\frac{\ell^2}{\ell^2+s} + \mathcal{O}(n^{s-1}),
\end{equation*} 
which in turn implies the convergence of the normalized moments. Note that 
\begin{equation*}
\prod_{\ell=1}^{m}\frac{\ell^2}{\ell^2+s}
=\frac{(m!)^2}{\Gamma(m+1-i\sqrt{s})\Gamma(m+1+i\sqrt{s})}\cdot \frac{\pi\sqrt{s}}{\sinh{\pi\sqrt{s}}},
\end{equation*}
where $i$ denotes the imaginary unit $i^2=-1$. 

\smallskip

Concerning case (2) such that $n\in\N$ is fixed and $m\to\infty$
we simply observe that the probability mass function
\begin{equation*}
\P\{X_{n,m}=k\}=\sum_{\ell=k}^{n}(-1)^{\ell-k}\binom{n}{\ell}\binom{\ell}{k}\prod_{j=1}^{m}\frac{j^2}{j^2+s}
\to \P\{Z_n=k\},
\end{equation*}
due to 
\begin{equation*}
\prod_{\ell=1}^{m}\frac{\ell^2}{\ell^2+s} \xrightarrow[m\to\infty]{} \prod_{\ell=1}^{\infty}\frac{\ell^2}{\ell^2+s}=\frac{\pi\sqrt{s}}{\sinh{\pi\sqrt{s}}}.
\end{equation*}
The moment convergence follows immediately from the explicit expression of the ordinary moments of $X_{n,m}$.
For the final case (3) $\min\{m,n\}\to\infty$ we check the convergence of the moments using the explicit expression~\eqref{GENSAMPfinalmom}
\begin{equation*}
\E\Big(\frac{X_{n,m}}{n}\Big)^s \xrightarrow[\min\{n,m\}\to\infty]{}\,\,\E(W^s)=\frac{\pi\sqrt{s}}{\sinh{\pi\sqrt{s}}}.
\end{equation*}
By Carleman's criterion the moments define a unique distribution. In order to identitfy the distribution function 
we observe, using Stirling's formula~\eqref{GENSAMPStirling}, that the probability mass function 
$\P\{X_{n,m}=k\}$ satisfies a local limit law of the form
\begin{equation*}
n\P\{X_{n,m}=k\} = \Big(2\sum_{j=1}^{\infty}(-1)^{j-1}j^2 q^{j^2-1}\Big)(1+o(1))
=\Theta'(q)(1+o(1)).
\end{equation*}
for $k=\lfloor n\cdot q\rfloor$ with $q\in(0,1)$, where $\Theta(q)=\sum_{n\in\Z}(-1)^nq^{n^2}=1+2\sum_{n\ge 1}(-1)^nq^{n^2}$ denotes 
the Jacobi Theta function. This implies after a few routine calculations, 
consisting of deriving the asymptotics of the distribution function
$\P\{X_{n,m}\le qn\}=\sum_{\ell=0}^{qn}\P\{X_{n,m}=\ell\}$, the convergence in distribution of the random variable
$X_{n,m}/n$ to a random variable $W$ with distribution function $1-\Theta(q)$ and support $[0,1]$.
It is not immediately obvious that $\lim_{q\to1^{-}}\P\{W\le q\}=1$. However, by using the Jacobi Triple-product-identity 
\begin{equation*}
\Theta(q)=\sum_{n\in\Z}(-1)^nq^{n^2}=1+2\sum_{n\ge 1}(-1)^nq^{n^2}=\prod_{j=1}^{\infty}(1-q^{2j})(1-q^{2j-1})^2,
\end{equation*}
we see that $\P\{W\le q\} = 1-\Theta(q)=1-\prod_{j=1}^{\infty}(1-q^{2j})(1-q^{2j-1})^2$, such that $1-\Theta(q)$ is indeed a well 
defined distribution function with support $[0,1]$. Concerning Corollary~\ref{GENSAMPtheTheta2} we simply observe that 
\begin{equation*}
\frac{Z_n}{n}\xrightarrow[n\to\infty]{d} W,\quad\text{and also}\quad Y_m\xrightarrow[m\to\infty]{d} W,
\end{equation*}
which can be easily checked from the explicit expressions of the moments.
\end{proof}

\begin{proof}[Proof of Corollary~\ref{GENSAMPtheTheta3}]
Let $\epsilon\law \exp(\lambda)$. The Laplace transform $\E(e^{-s\epsilon})$ of $\epsilon$ is given by $\E(e^{-s\epsilon})=\frac{1}{1+\frac{s}{\lambda}}$. 
Let $(\epsilon_{\ell})_{\ell\in\N}$ be i.i.d. $\Exp(1)$ distributed random variables. 
Using the fact that $c\cdot\epsilon\law \Exp(\frac{1}{c})$, we obtain
\begin{equation*}
\E\Big(\exp\big(-s\cdot\sum_{\ell=1}^{m}\epsilon_{\ell}c_{\ell}\big)\Big)=\prod_{\ell=1}^{m}\E(e^{-s\epsilon_{\ell}c_{\ell}})
=\prod_{\ell=1}^{m}\frac{1}{1+c_{\ell} s}.
\end{equation*}
Since the moments random variable $Y_m$ and $W$ are given by
\begin{equation*}
\E(Y_m^s)=\prod_{\ell=1}^{m}\frac{\ell^2}{\ell^2+s},\quad 	\E(W^s)= \frac{\pi\sqrt{s}}{\sinh{\pi\sqrt{s}}}
=\prod_{\ell=1}^{\infty}\frac{\ell^2}{\ell^2+s},
\end{equation*}
we obtain the stated decompositions,
\begin{equation*}
Y_m\law \exp\Big(-\sum_{\ell=1}^{m}\frac{\epsilon_{\ell}}{\ell^2}\Big), \quad
W\law \exp\Big(-\sum_{\ell=1}^{\infty}\frac{\epsilon_{\ell}}{\ell^2}\Big).
\end{equation*}
\end{proof}
Note that such decompositions are closely related to Brownian excursions and Zeta functions, see the work of Biane, Pitman and Yor~\cite{Biane}.

There exist fully analogous results for closely related weight sequences, also involving functions related to the Jacobi Theta function. We want to point out two particular choices, namely $A=(\alpha_n)_{n\in\N}=(n)_{n\in\N}$, $B=(\beta_m)_{m\in\N}=(\binom{m+1}{2})_{m\in\N}$, and $A=(n)_{n\in\N}$, $B=((m-\frac12)^2)_{m\in\N}$ For example, choosing the beforehand mentioned weight sequences $A=(n)_{n\in\N}$ and $B=(\binom{m+1}{2})_{m\in\N}$, the random variable $X_{n,m}/n$ converges for $\min\{m,n\}\to\infty $ in distribution
and with convergence of all integer moments to a random variable $W$ with support $[0,1]$, such that 
\begin{equation*}
\P\{W\le q\}= 1-\sum_{\ell\ge 0}(-1)^{\ell}(2\ell+1)q^{\binom{\ell+1}{2}},\qquad\text{with moments}\quad \E(W^s)=\frac{2s\pi}{\cosh\big(\frac{\pi}{2}\sqrt{8s-1}\big)}.
\end{equation*}
Using an identity of Jacobi for the cube of the Euler function $\phi(q)=\prod_{n=1}^{\infty}(1-q^n)$, namely 
\begin{equation*}
\phi^3(q)=\sum_{\ell\ge 0}(-1)^{\ell}(2\ell+1)q^{\binom{\ell+1}{2}}, 
\end{equation*}
one can observe that $\P\{W\le q\}=1-\phi^3(q)$ is indeed a well defined distribution function satisfying $\P\{W\le 1\}=1$. 
For the weight sequences $A=(n)_{n\in\N}$ and $B=((m-\frac12)^2)_{m\in\N}$ one encounters a similar random variable $W$ with support $[0,1]$, 
\begin{equation*}
\P\{W\le q\}= \frac{4}{\pi}\sum_{\ell\ge 1}(-1)^{\ell}\,\frac{q^{ (\ell-\frac12)^2}}{2\ell-1},\qquad\text{with moments}\quad \E(W^s)=\frac{1}{\cosh\big(\pi\sqrt{s}\big)}.
\end{equation*}
\begin{remark}
Since the product $\prod_{\ell=1}^{m}\ell^r/(\ell^r+s)$ converges for arbitrary real $r>1$, we expect that the results of Theorem~\ref{GENSAMPtheTheta1} and Corollary~\ref{GENSAMPtheTheta2} can be extended at least to weight sequences of the form $(n)_{n\in\N}$, $(m^r)_{m\in\N}$, for $r>1$.  Moreover, decompositions into sums of exponentially distributed random variables,
similar to Corollary~\ref{GENSAMPtheTheta3}, are expected to appear, which are again closely related to the distributions discussed by Biane, Pitman and Yor~\cite{Biane}.
\end{remark}

\section{Conclusion and Outlook}
We have studied two urn models with general weights, generalizing two well known P\'olya-Eggenberger urn models, namely the sampling without replacement urn 
and the OK Corral urn model. Assuming that the urn contains two types of balls, we derived explicit results for the distribution of the number of white balls, when all black
have been drawn. As a byproduct, we obtained results for the aforehand mentioned P\'olya-Eggenberger urn models. 
We also studied higher dimensional urn models and managed to obtain again explicit expressions. Furthermore, we have shown 
that the two urn models studied in this paper are essentially the same, since they obey a duality relation, also valid for the higher dimensional urn models. 
Examplarily, we discussed in the present work a particular example of urn model I with weight sequences $A=(n)_{n\in\N}$, $B=(m^2)_{m\in\N}$ and 
obtained a detailed description of the limit laws, curiously involving the the Jacobi Theta function $\Theta(q)=\sum_{n\in\Z}(-1)^nq^{n^2}$.
Since both urn models can be handled together, the author is currently studying the various limit laws arising in the Sampling without replacement/OK Corral urn models,
trying to unify the previous results in literature. In particular, the weight sequences $A=(n^a)_{n\in\N}$, $B=(m^b)_{m\in\N}$, with $a,b\in\R\setminus\{0\}$ lead to
a variety of different limiting distributions, some of which are closely related to Brownian excursions and the exponential constructions of Biane, Pitman and Yor~\cite{Biane}.

\vspace{-0.2cm}

\section*{Acknowledgements}
The author warmly thanks Alois Panholzer for introducing him to this topic, and for several 
valuable and interesting discussions. Moreover, he thanks Svante Janson for pointing out the decompositions 
of random variables $Y_m$ and $W$.

\end{document}